\documentclass[10pt]{amsart}
\usepackage{amsmath, amssymb}
\usepackage{mathrsfs}
\newcommand{\no}[1]{#1}
\renewcommand{\no}[1]{}
\no{\usepackage{times}\usepackage[subscriptcorrection, slantedGreek, nofontinfo]{mtpro}
\renewcommand{\Delta}{\upDelta}}
\usepackage{color}
\usepackage{enumerate,comment}
\usepackage{graphicx}
\usepackage{dsfont}
\usepackage{epstopdf}
\graphicspath{{./pics/}}
\usepackage{booktabs}
\usepackage{threeparttable}

\date{\today}
\newcommand{\bel}{\begin{equation} \label}
\newcommand{\ee}{\end{equation}}
\def\beq{\begin{equation}}
\def\eeq{\end{equation}}
\newcommand{\bea}{\begin{eqnarray}}
\newcommand{\eea}{\end{eqnarray}}
\newcommand{\beas}{\begin{eqnarray*}}
\newcommand{\eeas}{\end{eqnarray*}}

\newcommand{\re}{\mathfrak R}

\newcommand{\R}{\mathbb{R}}

\newtheorem{Thm}{Theorem}[section]
\newtheorem{lem}{Lemma}[section]
\newtheorem{definition}{Definition}[section]

\newtheorem{prop}{Proposition}[section]

\newtheorem{example}{Example}[section]
\numberwithin{equation}{section}

\renewcommand{\d}{\mathrm{d}}
\allowdisplaybreaks
\providecommand{\abs}[1]{\left\lvert#1\right\rvert}
\providecommand{\norm}[1]{\left\lVert#1\right\rVert}

\def\phi {\varphi}

\title[]
{Stability Estimates for the Inverse Problem of Reconstructing Point sources in Parabolic Equations}

\author{Kuang Huang$^1$\and Bangti Jin$^1$ \and Yavar Kian$^2$ \and Faouzi Triki$^3$}

\date{}

\begin{document}

\begin{abstract}
In this work, we investigate the stability issue of the inverse problem of determining the locations and time-dependent amplitudes of point sources in a parabolic equation  with a non-self adjoint elliptic operator from boundary observations. We derive different stability estimates for determining  the locations and the amplitudes of the sources in the space, the plane as well as in dimension one. The analysis employs a novel approach that combines several different arguments, including the improved regularity of the solutions, the application of Carleman estimates, time extension of solutions, and construction of explicit solutions to the adjoint equations. Further we provide numerical reconstructions to complement the theoretical findings.\\
\textbf{Key words}: inverse source problem, point source identification, stability estimate, parabolic equation,  Carleman estimate.

\medskip
\noindent
{\bf Mathematics subject classification 2020 :} 35R30, 	35K20.
\end{abstract}
\maketitle

\renewcommand{\thefootnote}{\fnsymbol{footnote}}
\footnotetext{\hspace*{-5mm}
\begin{tabular}{@{}r@{}p{16cm}@{}}
& Manuscript last updated: \today.\\
$^1$
& Department of Mathematics, The Chinese University of Hong Kong, Shatin, New Territories, Hong Kong, P.R. China.\\
&The work of K. Huang is supported by a start-up fund from The Chinese University of Hong Kong and that of B. Jin  by Hong Kong RGC General Research Fund (Project 14306423), Hong Kong RGC  ANR / RGC Joint
Research Scheme (A-CUHK402/24) and a
start-up fund from The Chinese University of Hong Kong. \\
& E-mail: kuanghuang@cuhk.edu.hk,  bangti.jin@gmail.com.\\
$^2$
& Univ Rouen Normandie, CNRS, Laboratoire de Math\'{e}matiques Rapha\"{e}l Salem, UMR 6085, F-76000 Rouen, France.\\
& The work of Y. Kian is supported by the French National Research Agency ANR and Hong Kong RGC Joint Research Scheme for the project IdiAnoDiff (grant ANR-24-CE40-7039). E-mail: yavar.kian@univ-rouen.fr\\
$^3$
& Laboratoire Jean Kuntzmann,
UMR CNRS 5224, Universit\'e Grenoble-Alpes, 150 Pl. du Torrent,
38400 Saint-Martin- d'H\`eres, France.  E-mail: faouzi.triki@univ-grenoble-alpes.fr.

\end{tabular}}

\section{Introduction}
Parabolic equations such as advection-diffusion equations can accurately describe the spatio-temporal evolution of a scalar concentration field, e.g., the concentration of a contaminant in air, water, or soil generated by the pollutant, and represent one prominent class of mathematical models in environmental fluid dynamics and pollutant transport studies. This class of models is widely employed in air quality forecasting, groundwater contamination studies, and oceanographic pollutant transport \cite{MVBT}. Thus their analysis under realistic environmental conditions is central to quantify pollutant exposure, assess ecological risks, and support evidence-based policy for emission control and environmental protection. One pivotal step in pollution management lies in accurately identifying pollution sources.
Thus, inverse source problems for the advection diffusion equation have emerged as a major research topic in environmental applications \cite{GK,HJL,JYZ}. Mathematically, this task often involves recovering point sources in the advection diffusion equation, i.e., to infer unknown characteristics of sources, e.g., the locations, intensities and the history of release, from sparse and noisy boundary data.

Let $\Omega\subset \mathbb{R}^d$ ($d=1,2,3$) be an open bounded and simply connected domain with a
$ C^{2}$ boundary $\partial \Omega$ and $T>0$. The advection diffusion problem is described by the following initial boundary value problem:
\begin{equation}\label{eq1}
\left\{\begin{aligned}
\partial_tu -\Delta u+A\cdot\nabla u+\mu u &=  \sum_{k=1}^N\lambda_k(t)\delta_{x_k}, \quad \mbox{in }\Omega\times(0,T),\\
 \partial_\nu u&= 0, \quad \mbox{on } \partial\Omega\times(0,T), \\
u&=u_0, \quad \mbox{in } \Omega\times \{0\},
\end{aligned}\right.
\end{equation}
where $A\in\R^d$ is a constant vector describing the velocity field of the moving pollutant, $\mu\in\R$ denotes the reaction term, and $u_0\in H^1(\Omega)$ denotes the initial distribution of the pollutant concentration. The diffusion of pollution is generated by $u_0$ and the point sources at locations $x_1,\ldots,x_N\in \Omega$, $N\in\mathbb N^*$, and the amplitudes $\lambda_1,\ldots,\lambda_N\in L^2(0,T)$. This structure of the source is justified by the physical nature of environmental pollutants which often emanate at varying rates from fixed locations, e.g., industrial sites and deposited waste materials. Due to the presence of Dirac sources, the solution $u\in L^2(0,T;L^2(\Omega))$ of problem \eqref{eq1} may be considered in the transposition sense \cite{LM1} and $u(x,t)$, $x\in\Omega$, $t\in(0,T)$, denotes the concentration of contaminant at position $x$ and time $t$. In this work, we investigate the stable
determination of the locations $x_1, \ldots, x_N$, and the amplitudes $\lambda_1, \ldots, \lambda_N$ of the sources of pollution from
the knowledge of the concentration $u(x, t)$, $(x, t) \in \partial\Omega\times (0, T )$.

The unique determination of point sources in parabolic equations has received much attention. El Badia and Ha Duong \cite{EH} addressed the unique recovery of static sources in dimension $d=2,3$, and later extended the result to the one-dimensional case  \cite{EHH}. Later Andrle et al \cite{ABE,AE} investigated the unique recovery of moving point sources. For the related problem of identifying initial point sources from the terminal time observation, see  \cite{CaKu,UYZ,CaZu}.
In sharp contrast, the stability analysis for point sources in parabolic problems faces significant technical challenges. The unique recovery of point sources in parabolic equations relies on two main ingredients: the regularity of the solution $u$ to problem \eqref{eq1} and unique continuation properties \cite{ABE,AE,EH,EHH}. However, the strategy does not enable the derivation of stability estimates. Similarly, the use of adjoint solutions, which has proven effective for the Helmholtz equation \cite{AET,BLT}, does not lead to global stability estimates in the parabolic setting \cite{QYY}. Moreover, due to the lack of regularity, it remains unclear how the classical Carleman estimate-based approach (see, e.g. \cite{BBCS,BKS,Cho,IY}), which is effective for the stable recovery of space-dependent sources in $L^2(\Omega)$, can be extended to the case of point sources.

To the best of our knowledge, there are only two works on the stable recovery of time-dependent point-sources in the diffusion equation. The first one is due to Komornik and Yamamoto  \cite[Theorem 3.1]{KoY}, who investigated the stable recovery of locations of multiple sources, with known identical amplitudes $\lambda(t)$, from the final time  over-determination under suitable nonvanishing conditions on $\lambda$, and proved a Lipschitz stability estimate. The second work is the recent contribution due to Qiu et al \cite{QYY}, who investigated the stable recovery of point values of the amplitude. The global stable recovery of both amplitudes and locations of general time-dependent point sources in parabolic equations remains an open problem. Note that in addition to its mathematical interest, the stability issue is  motivated by providing mathematical guarantees for relevant reconstruction algorithms \cite{HuJinZhou:2025}.

In this work, we establish several new stability results for the simultaneous recovery of both the locations and the amplitude of the point sources, and take first steps towards addressing the stability issue of parabolic point source identification in the general case. Our theoretical findings indicate that the determination of the location is Lipschitz / H\"{o}lder stable, similar to the elliptic counterpart \cite{ElBadiaElHajj:2012,ElBadiaElHajj:2013,BLT,Ve}, but the determination of the amplitude $\lambda$ is only logarithmically stable. The precise statements and relevant discussions are given in Section \ref{sec:main}. The analysis strategy combines several analytical tools, including Carleman estimates for parabolic equations, improved local regularity results, time extensions of the solutions to problem \eqref{eq1}, and the construction of explicit solutions to the adjoint elliptic equation. While some of these techniques have appeared separately in the literature in different contexts \cite{BLT,ChY,IY}, their combination in the context of point source recovery appears to be entirely novel. These theoretical findings will be complemented by numerical illustrations in Section \ref{sec:numer}. The numerical results indicate that the location $x$ can be more stably recovered than the amplitude $\lambda(t)$, as predicted by the stability analysis.

The rest of the paper is organized as follows. In Section \ref{sec:main}, we present and discuss the main theoretical results. In Section \ref{sec:prelim}, we present several preliminary results on the direct problem that are crucial for the analysis of the inverse problem. In Section \ref{sec:proof}, we give the proofs of the main results. Finally in Section \ref{sec:numer}, we present numerical results to illustrate the stable reconstruction of the point sources.

\section{Main results}\label{sec:main}

Throughout we assume that there exists $T_0\in(0,T)$ such that
\bel{lambda}\lambda_k(t)=0,\quad t\in(T_0,T),\ k=1,\ldots,N.\ee
Physically this corresponds to the case that the sources become inactive after the time $t=T_0$. This assumption is crucial for the subsequent analysis and is fundamental for restricting the measurements to the lateral boundary $\partial\Omega \times (0,T)$. Similar assumptions have been imposed in many works dealing with the recovery of point sources, both in terms of uniqueness and stability \cite{EH,QYY}. While the assumption can be removed for uniqueness results \cite{ABE,AE}, it remains unclear how it can be avoided in the context of stability. Fix $T_1\in(T_0,T)$ and let $B(y,r):=\{x\in\R^d:\ |x-y|<r\}$ denote the ball centered at $y$ with a radius $r$. Fix also $r_0>0$ such that
$\cup_{k=1}^N \overline{B(x_k,r_0)}\subset \Omega$
and, for all $r\in(0,r_0)$, let
$\Omega_r:=\Omega\setminus\cup_{k=1}^N \overline{B(x_k,r)}$. We shall prove in Theorem \ref{t4} that, for the unique solution $u\in L^2(0,T;L^2(\Omega))$ in the transposition sense \cite{LM1} (see also Definition \ref{def:transposition}) of problem \eqref{eq1},
we have
\begin{equation*} u|_{\Omega_r\times (0,T)}\in L^2\left(0,T; H^2(\Omega_r)\right)\cap H^1(0,T;L^2(\Omega_r))\quad\mbox{and}\quad u|_{\Omega\times (T_1,T)}\in H^1(T_1,T;H^2(\Omega)).
\end{equation*}
Using these improved local regularity properties, we obtain four main theoretical results on point source identification depending on the space dimension $d$ and the number $N$ of sources.

\subsection{Source recovery in the space}
First we consider the  stability issue of the inverse problem of identifying one point source in the space, i.e., $d=3$ and $N=1$.
\begin{Thm}\label{t2} Let $d=3$ and, for $j=1,2$, let $\lambda^j\in L^2(0,T)$, satisfy condition \eqref{lambda} with $N=1$, $\lambda^j\not\equiv0$,  $x^j\in\Omega$ and $u_0\in H^1(\Omega)$. Assume also that there exists $s\in(0,\frac{1}{2})$ such that $\lambda^1\in H^s(0,T)$
and let  the following conditions
\begin{align}\label{t2b}\lambda=\lambda^1-\lambda^2\in H^1_0(0,T)\quad &\textrm{and}\quad \exists M>0,\ \norm{\lambda^1}_{H^s(0,T)}+\norm{\lambda}_{H^1(0,T)}\leq M,\\
\label{t2ab}\exists &r_\star>0,\ \norm{\lambda^1}_{L^2(0,T)}\geq r_\star
\end{align}
be fulfilled. Let $S$ be an open nonempty subset of the boundary $\partial\Omega$ and let $u^j$, $j=1,2$, be the solution of problem \eqref{eq1} with $N=1$, $\lambda_1=\lambda^j$ and $x_1=x^j$. Then there exists a constant $C>0$  depending  on $r_\star$, $s$, $\Omega$, $T_1$, $A$, $\mu$, $S$, $M$ and $T$ such that
\begin{align} \label{t2ab1}
|x^1-x^2|&\leq C(\norm{u^1-u^2}_{H^1(S\times(T_1,T))}+\norm{u^1-u^2}_{L^2(0,T;L^2(\partial\Omega))}),\\
\label{t2ab2}\norm{\lambda^1-\lambda^2}_{L^2(0,T)}&\leq C\ln\big( 3+(\norm{u^1-u^2}_{H^1(S\times(T_1,T))}+\norm{u^1-u^2}_{L^2(0,T;L^2(\partial\Omega))})^{-1}\big)^{-2}.
\end{align}
\end{Thm}

To the best of our knowledge, in Theorem \ref{t2}, we obtain the first stability estimates for the simultaneous recovery of both the location $x$ and amplitude $\lambda$ of the point source in the parabolic equation. We are only aware of similar results for hyperbolic and Helmholtz equations \cite{AET,BLT,ElBadiaElHajj:2013}. For parabolic equations, stability estimates for the inverse problem  were limited to results for sources with a known and identical amplitude in \cite{KoY} or some partial stability estimates in \cite{QYY} (only in the 2D case). It appears that even the stable recovery of the locations with unknown amplitudes, stated in estimate \eqref{t2ab1}, is new. In Theorem \ref{t2}, we have extended these results for the first time to the more realistic situation where both locations and the amplitudes of the sources are unknown. In Theorem \ref{t2}, the recovery of the source location $x$ satisfies the Lipchitz stability estimate \eqref{t2ab1}, which is comparable to what was known for elliptic equations \cite{BLT} (for one point source). However, the recovery of the amplitude $\lambda(t)$ satisfies only a logarithmic stability estimate \eqref{t2ab2} which indicates a strong instability of the inverse problem for parabolic equations that was not observed earlier for elliptic and hyperbolic equations. These theoretical findings are also confirmed by the numerical experiments in Section \ref{sec:numer}: The recovery of the location is far more stable than that of the amplitudes (in the one- and two-dimensional cases).

The proof of Theorem \ref{t2} combines several different tools including Carleman estimates for parabolic equations, regularity properties and time extension of solutions of problem \eqref{eq1}, and the construction of explicit solutions of the adjoint elliptic equation. Although some of these arguments can be found separately in the literature in different contexts \cite{BLT,ChY,IY}, the combination of all of these tools appears to be completely new in the context of point-source recovery.

In Theorem \ref{t2}, we have restricted the analysis to one source (i.e. $N=1$). We believe that, by following the argument in \cite{BLT}, one can extend, under suitable assumptions, the stability estimates to the case of multiple point sources. We believe that, this improvement will come at the expense of  a weaker stability estimate and we leave this interesting issue to future investigations. Similarly,  weaker logarithmic stability estimates would be expected from the measurements restricted to an arbitrary subset of the boundary. Moreover, we have focused on the case of constant coefficients. The extension to  variable coefficients would require new technical tools, which we leave to further investigations.

\subsection{Source recovery in the plane}
The second and third results consider the stable identification of point sources in the plane (i.e., $d=2$) for  both situations $N\geq2$ and $N=1$. We first consider the recovery of the locations of multiple point sources.

\begin{Thm}\label{t3} Let $d=2$, $\mu=-\frac{|A|^2}{4}$ and, for $j=1,2$, we fix $\lambda^j_1,\ldots,\lambda^j_N\in L^2(0,T)$, satisfying \eqref{lambda} with  and $\lambda_k=\lambda^j_k$, $k=1,\ldots,N$,  $x^j_1,\ldots,x_N^j\in\Omega$, $u_0\in H^1(\Omega)$. Assume that $\lambda_k^j\not\equiv0$, $k=1,\ldots,N$, $j=1,2$, and the following extra assumption
\bel{t3a}\exists r_2>0,\quad \min_{k=1,\ldots,N}\abs{\int_0^T\lambda_k^1(t)\d t}\geq r_2.\ee
Assume also that there exist $\delta_0>0$ and a permutation $\sigma$ of $\{1,\ldots,N\}$ such that
\bel{t3c}k,\ell\in\{1,\ldots,N\},\ k\neq\ell\Rightarrow |x^j_k-x_\ell^j|\geq\delta_0,\ j=1,2,\ee
\bel{t3d}|x_k^1-x_{\sigma(k)}^2|=\min_{\ell=1,\ldots N}|x_k^1-x_\ell^2|,\quad k=1,\ldots,N.\ee
Let $S$ be an open nonempty subset of the boundary $\partial\Omega$ and let $u^j$, $j=1,2$, be the solution in the transposition sense of problem \eqref{eq1} with  $\lambda_k=\lambda_k^j$ and $x_k=x_k^j$, $k=1,\ldots,N$. Then,  there exists a constant $C>0$  depending  on $r_2$, $\delta_0$, $S$, $\Omega$, $A$, $\mu$, $T_1$ and $T$   such that
\bel{t3e}\max_{k=1,\ldots,N} |x_k^1-x_{\sigma(k)}^2| \leq C(\norm{u^1-u^2}_{H^1(S\times(T_1,T))}+\norm{u^1-u^2}_{L^2(0,T;L^2(\partial\Omega))})^{\frac{1}{N}}.\ee

\end{Thm}
Now assuming that $N=1$, in a similar way to Theorem    \ref{t2}, we show the stable recovery of both the location $x$ and the amplitude $\lambda(t)$ of one point source.
\begin{Thm}\label{t100} Let $d=2$, $N=1$ and let the condition of Theorem \ref{t2} be fulfilled. Then, the estimates \eqref{t2ab1} and \eqref{t2ab2} hold true.
\end{Thm}

In addition to extending the results of Theorem \ref{t2} to the 2D case, we obtain in Theorem \ref{t3} the stable identification of locations of multiple point sources with a H\"older stability estimate, with the H\"{o}lder exponent agreeing with that for the elliptic problem \cite{ElBadiaElHajj:2012,ElBadiaElHajj:2013,BLT}. In the proof, we exploit properties of harmonic functions in the plane for building a specific class of solutions,
of parabolic equations with constant coefficients, that we use for the derivation of the estimate \eqref{t3e}. This approach is limited to the two-dimensional case and cannot be applied to dimension $d=1,3$. Moreover, it is based on condition \eqref{t3a}, which is stronger than condition \eqref{t2ab}, but is fulfilled for a large class of parameters $\lambda_k^1$, $k=1,\ldots,N$, e.g., functions of constant sign. The stable recovery of both amplitudes and location of the source are stated in Theorem \ref{t100}, which is similar to Theorem \ref{t2}. In contrast to the locations of multiple sources of Theorem \ref{t3}, we limit the stable recovery of the amplitude $\lambda(t)$ to one point source $N=1$.

The condition $\mu=-\frac{|A|^2}{4}$ in  Theorem \ref{t3} is a technical restriction related to the method used for simultaneously recovering the locations of multiple sources with unknown amplitudes, and it is not needed in the case of one single point source of Theorem \ref{t100}. The number of points in the two candidate point sources in  Theorem \ref{t3} is assumed to be identical so that there exists a permutation $\sigma$ of the set $\{1,\ldots,N\}$ satisfying condition \eqref{t3c}.
\subsection{Source recovery in an interval}
Finally, we consider the one-dimensional case. The results can be stated as follows.

\begin{Thm}\label{t5} Let $d=1$,  $\mu=-\frac{|A|^2}{4}$, $\Omega=(0,\ell)$ and, for $j=1,2$, we fix $\lambda^j_1\in L^2(0,T)$, satisfying condition \eqref{lambda} with   $\lambda=\lambda^j_1$,  $x^j_1\in\Omega$, $u_0\in H^1(\Omega)$. Assume that $\lambda_1^j\not\equiv0$, $j=1,2$, and, \eqref{t2b} and \eqref{t3a} are fulfilled with $N=1$. Let $u^j$, $j=1,2$, be the solution of problem \eqref{eq1} with $N=1$, $\lambda_j=\lambda_1^j$ and $x_j=x_1^j$. Then,  there exists a constant $C>0$  depending  on $r_2$,  $\ell$, $A$, $M$, $\mu$ and $T$   such that
\begin{align}\label{t5b}|x_1^1-x_1^2|&\leq C\Big(\norm{u^1(\ell,\cdot)-u^2(\ell,\cdot)}_{H^1((T_1,T))}+\sum_{j=0}^1\norm{u^1(j\ell,\cdot)-u^2(j\ell,\cdot)}_{L^2(0,T)}\Big),\\
\label{t5c}
\norm{\lambda^1_1-\lambda^2_1}_{L^2(0,T)}
&\leq C\ln\Big( 3+\big(\norm{u^1(\ell,\cdot)-u^2(\ell,\cdot)}_{H^1((T_1,T))}+\sum_{j=0}^1\norm{u^1(j\ell,\cdot)-u^2(j\ell,\cdot)}_{L^2(0,T)}\big)^{-1}\Big)^{-2}.
\end{align}
\end{Thm}

The proof of Theorem \ref{t5} is mostly based on suitable adaptation of the argument of Theorem \ref{t2}  to the one-dimensional case. The results are stated for one point source, which cannot be further relaxed since even for uniqueness results,  there are counterexamples for the recovery of multiple point sources  from boundary / partial interior measurements when $d=1$ (see e.g., \cite{ABE,HuangJinKian:2025}).

\section{Preliminary properties} \label{sec:prelim}
In this section we present several preliminary properties of solutions of problem \eqref{eq1} in the transposition sense \cite[Chapter 3, Section 2]{LM1} which will play a fundamental role in the analysis. By Sobolev embedding theorem \cite[Theorem 4.12, p. 85]{AdamsFournier:2003}, $H^2(\Omega)$ embeds continuously into $C(\overline{\Omega})$ for $d = 1,2,3$. Thus, the maps $\{\delta_{x_k}
 \}_{k=1}^N$ are continuous linear forms on $H^2(\Omega)$ and for any $v\in H^2(\Omega)$, we have $\delta _{x_k}(v) = v(x_k)$, $k =1,...,N$. This motivates the following definition of solutions in the transposition sense.
 \begin{definition}\label{def:transposition}
 A function $v \in L^2(0,T;L^2(\Omega))$ is said to be a solution of problem \eqref{eq1} in the transposition sense if for all $w \in  H^1(0,T;L^2(\Omega)) \cap L^2(0,T;H^2(\Omega)\cap H^1_0
(\Omega))$, satisfying $w|_{\Omega\times\{T\}}\equiv 0$ and $\partial_\nu w+(A\cdot \nu)w=0$ on $\partial\Omega\times(0,T)$, we have
\begin{equation*}
    \int_0^T\!\!\int_\Omega v(x,t)[-\partial_t w(x,t) -\Delta w(x,t) - A\cdot\nabla w(x,t)+\mu w(x,t) ]\d x\d t = \sum_{k=1}^N \int_0^T \lambda_k(t)w(x_k,t)\d t.
\end{equation*}
 \end{definition}

The unique existence of a solution $u$ to problem \eqref{eq1} in the sense of transposition follows directly from Riesz representation theorem.
The next result includes  regularity properties and the estimate from the boundary $\partial\Omega$ of values at the final time $t=T$ of solutions of problem \eqref{eq1}.

\begin{Thm}\label{t4} Let $u\in L^2(0,T;L^2(\Omega))$ be the solution in the transposition sense of problem \eqref{eq1} and let condition \eqref{lambda} be fulfilled. Then, for all $r\in(0,r_0)$,
we have
$$
u|_{\Omega_r\times (0,T)}\in L^2\left(0,T; H^2(\Omega_r)\right)\cap H^1(0,T;L^2(\Omega_r))\quad\mbox{and}\quad u|_{\Omega\times (T_1,T)}\in H^1(T_1,T;H^2(\Omega)),
$$
with
\bel{t4aa}\partial_tu(x,t)-\Delta u(x,t)+A\cdot\nabla u(x,t)+\mu u(x,t)=0,\quad (x,t)\in \Omega_r\times (0,T)\cup \Omega\times (T_1,T).\ee
 Moreover, for an open nonempty subset $S$  of the boundary $\partial\Omega$,  the following estimate holds
\bel{t4a} \norm{u(\cdot,T)}_{H^2(\Omega)}\leq C\norm{u}_{H^1(S\times(T_1,T))},\ee
with the constant $C>0$ depending on $\Omega$, $T$, $T_1$, $A$, $\mu$ and $S$.
\end{Thm}

In order to prove Theorem \ref{t4}, we first recall several preliminary properties related to Carleman estimates for solutions of parabolic equations. More precisely, in view of \cite[Lemma 2.3]{IY} and \cite[Lemma 1.2]{Im}, there exists a weight function $\psi \in {C}^2(\overline{\Omega})$ such that the following properties are fulfilled: (i) $\psi (x) > 0$ for all $x \in \Omega$;
(ii)  $\vert\nabla \psi (x)| >0 $ for all $x \in \overline{\Omega}$;
(iii) $\partial_{\nu} \psi (x) \leqslant 0 $ for all $x \in \partial \Omega \backslash S$.

Next we introduce the following weight functions for every $\rho \in (0,+\infty)$,
$$g(t)=(t-T_1)(T-t),\ \phi(x,t)=\frac{e^{\rho\psi(x)}}{g(t)},\ \alpha(x,t)=\frac{e^{\rho\psi(x)}-e^{-2\rho\norm{\psi}_{L^\infty(\Omega)}}}{g(t)},\quad (x,t)\in \overline{\Omega}\times(T_1,T).$$

Following \cite[Lemma 2.4]{IY}, we have the following Carleman estimate.
\begin{prop}\label{p1}
Let $v \in H^1(T_1,T;H^2(\Omega)) $.
Then there exists $\rho_0 \in (0,+\infty)$ such that for any $\rho\in [\rho_0,+\infty)$, there exists
$s_0 \in (0,+\infty)$, depending only
 $\Omega$, $S$, $T$, $T_1$, $A$, $\mu$ and $\rho$, such that the following estimate holds for  all $s \in [s_0,+\infty)$
\begin{align}
\label{p1aa}&\int_{T_1}^T\!\!\int_\Omega \left((s\phi)^{-1}\left( |\partial_tv|^2+\left|\sum_{i,j=1}^3\partial_{x_i}\partial_{x_j}v\right|^2\right)+s\phi|\nabla v|^2+(s\phi)^3 v^2\right)e^{2s\alpha}\d x\d t\\
\leq &C\int_{T_1}^T\!\!\int_\Omega g(t)^2 |\partial_tv-\Delta v+A\cdot\nabla v+\mu v|^2e^{2s\alpha}\d x\d t\nonumber\\
&+C\int_{T_1}^T\!\!\int_{S}((\partial_tv)^2+(s\phi)|\nabla v|^2+(s\phi)^3v^2)e^{2s\alpha}\d\sigma(x)\d t,\nonumber
\end{align}
for some positive constant $C$, depending only on  $\Omega$, $S$, $A$, $\mu$, $T$, $\rho$ and $s_0$.
\end{prop}
\begin{proof}
Using the weight function $\psi\in C^2(\overline{\Omega})$, we define the functions $\tilde{v}$, $\tilde{g}$, $\tilde{\phi}$ and $\tilde{\alpha}$ on $\Omega\times (0,T-T_1)$ by
\begin{align*}
\tilde{v}(\tau,x)&=v(x,T_1+\tau),\quad (x,\tau)\in\Omega\times (0,T-T_1),\\
\tilde{g}(\tau)&=\tau (T-T_1-\tau),\ \tilde{\phi}(x,\tau)=\frac{e^{\rho\psi(x)}}{\tilde{g}(\tau)},\ \tilde{\alpha}(x,\tau)=\frac{e^{\rho\psi(x)}-e^{-2\rho\norm{\psi}_{L^\infty(\Omega)}}}{\tilde{g}(\tau)},\quad (x,\tau)\in\Omega\times (0,T-T_1).
\end{align*}
By \cite[Lemma 2.4]{IY}, there exists
$s_0=s_0(\rho) \in (0,+\infty)$, depending only
 $\Omega$, $S$, $T$, $T_1$, $A$, $\mu$ and $\rho$, such that the following estimate
\begin{align}\label{p1ab}&\int_{0}^{T-T_1}\!\!\!\!\int_\Omega \left((s\tilde{\phi})^{-1}\left( |\partial_\tau\tilde{v}|^2+\left|\sum_{i,j=1}^3\partial_{x_i}\partial_{x_j}\tilde{v}\right|^2\right)+s\tilde{\phi}|\nabla\tilde{v}|^2+(s\tilde{\phi})^3 \tilde{v}^2\right)e^{2s\tilde{\alpha}}\d x\d \tau\\
\leq &C\int_{0}^{T-T_1}\!\!\!\!\int_\Omega \tilde{g}(t)^2 |\partial_\tau\tilde{v}-\Delta \tilde{v}+A\cdot\nabla \tilde{v}+\mu \tilde{v}|^2e^{2s\tilde{\alpha}}\d x\d \tau\nonumber\\
&+C\int_{0}^{T-T_1}\!\!\!\!\int_{S}((\partial_\tau\tilde{v})^2+(s\tilde{\phi})|\nabla \tilde{v}|^2+(s\tilde{\phi})^3\tilde{v}^2)e^{2s\tilde{\alpha}}\d\sigma(x)\d \tau\nonumber\end{align}
holds for  all $s \in [s_0,+\infty)$ and some positive constant $C$, depending only on  $\Omega$, $S$, $A$, $\mu$, $T$, $\rho$ and $s_0$. Finally by observing the relations
$$\tilde{g}(\tau)=g(T_1+\tau),\ \tilde{\phi}(x,\tau)=\phi(x,T_1+\tau),\ \tilde{\alpha}(x,\tau)=\alpha(x,T_1+\tau),\quad (x,\tau)\in\Omega\times (0,T-T_1),$$
and applying the change of variables $t=T_1+\tau$, we obtain the estimate \eqref{p1aa} from \eqref{p1ab}.
\end{proof}

Using Proposition \ref{p1}, we can prove Theorem \ref{t4}.
\begin{proof}
Without loss of generality, we may assume that $u_0\equiv0$. We divide the proof into three steps.

\noindent\textbf{Step 1.} In this step, we prove that, for all $r\in(0,r_0)$, we have the following improved local regularity
\begin{equation*}
u|_{\Omega_r\times (0,T)}\in L^2\left(0,T; H^2(\Omega_r)\right)\cap H^1(0,T;L^2(\Omega_r)).
\end{equation*}
For every $j\in\{1,\ldots,N\}$, define  $y_j$  on $\R^d\times[0,T]$ by
$$y_j(x,t)= \frac{1}{(4\pi)^{\frac{d}{2}}} e^{\frac{A}{2}\cdot(x-x_j)}\int_0^t \frac{e^{\tau_0(t-s)}\lambda_j(s) e^{-\frac{\vert x-x_j\vert^2}{4(t-s)}}}{(t-s)^{\frac{d}{2}}}\d s,\quad (x,t)\in\R^d\times[0,T],$$
with $ \tau_0=-\frac{|A|^2}{4}-\mu$.
It is well known that $y_j|_{\Omega\times(0,T)}\in L^2(0,T;L^2(\Omega))$, $y_j\in C^{\infty}((\R^d\setminus\{x_j\})\times[0,T])$ and, for all $w\in C^2(\overline{\Omega}\times[0,T])$ satisfying $w|_{\overline{\Omega}\times\{T\}}\equiv0$ and $(\partial_\nu w+(A\cdot\nu) w)|_{\partial\Omega\times[0,T]}\equiv0$, we have
\begin{equation}\label{trans11}
\int_0^T\!\!\int_\Omega y_j(-\partial_tw-\Delta w-A\cdot\nabla w+\mu w)\d x\d t=\int_0^T\lambda_j(t)w(x_j,t)\d t+\int_0^T\!\!\int_{\partial\Omega} w(x,t)\partial_\nu y_j(x,t)\d \sigma(x)\d t.
\end{equation}
By combining the continuous embedding of $H^2(\Omega)$ into $C(\overline{\Omega})$ (for $d=1,2,3$) with a density argument, we can prove that the identity \eqref{trans11} holds for any
$w\in H^1(0,T;L^2(\Omega))\cap L^2(0,T;H^2(\Omega))$ satisfying $w|_{\Omega\times\{T\}}\equiv0$ and $(\partial_\nu w+(A\cdot\nu) w)|_{\partial\Omega\times(0,T]}\equiv0$.
Now fix $g$ defined on $\partial\Omega\times[0,T]$ by
$$g(x,t)=-\sum_{j=1}^N\partial_\nu y_j(x,t),\quad (x,t)\in \partial\Omega\times[0,T],$$
and notice that $g\in C^2(\partial\Omega\times[0,T])$
with $\partial_t^kg(x,0)=0$, $k=0,1,2$, $x\in\partial\Omega$. Thus, in view of \cite[Theorem 6.1, Chapter 4]{LM2}, the initial boundary value problem
$$
\left\{\begin{aligned}
\partial_tv -\Delta v+A\cdot\nabla v+\mu v &= 0, \quad \mbox{in }\Omega\times(0,T],\\
 \partial_\nu v&= g, \quad \mbox{on } \partial\Omega\times(0,T], \\
v&=0, \quad \mbox{in } \Omega\times \{0\}
\end{aligned}\right.
$$
admits a unique solution $v\in H^1(0,T;L^2(\Omega))\cap L^2(0,T;H^2(\Omega))$. By combining this fact with the identity \eqref{trans11} and integrating by parts, we deduce that the map $u_\star\in L^2(0,T;L^2(\Omega))$, defined by
$$u_\star(x,t)=\sum_{j=1}^Ny_j(x,t)+v(x,t),\quad (x,t)\in \Omega\times(0,T],$$
satisfies the condition
$$
\int_0^T\!\!\int_\Omega u_\star(-\partial_tw-\Delta w-A\cdot\nabla w+\mu w)\d x\d t=\sum_{j=1}^N\int_0^T\lambda_j(t)w(x_j,t)\d t
$$
for any
$w\in H^1(0,T;L^2(\Omega))\cap L^2(0,T;H^2(\Omega))$ such that $w|_{\Omega\times\{T\}}\equiv0$ and $(\partial_\nu w+(A\cdot\nu) w)|_{\partial\Omega\times(0,T]}\equiv0$.
Therefore, $u_\star$ is a solution in the transposition sense of problem \eqref{eq1} and by the uniqueness of solutions in the transposition sense \cite[Theorem 12.1, Chapter 4]{LM2}, we have $u=u_\star$. Thus it follows that, for any  $r\in(0,r_0)$,
$$u|_{\Omega_r\times (0,T)}=\sum_{j=1}^Ny_j|_{\Omega_r\times (0,T)}+v|_{\Omega_r\times (0,T)}\in L^2\left(0,T; H^2(\Omega_r)\right)\cap H^1(0,T;L^2(\Omega_r)).$$

\noindent\textbf{Step 2.} In this step we will show that $u|_{\Omega\times (T_1,T)}\in H^1(T_1,T;H^2(\Omega))$. To this end, fix $h\in C^\infty([0,T];[0,1])$ such that,  $h=0$ on $[0,T_0+\delta]$ and $h=1$ on $[T_0+3\delta,T]$, with $\delta=(T_1-T_0)/10$. Let $w(x,t)=h(t)u(x,t)$ for $(x,t)\in \Omega\times(0,T]$. Then $w$ solves in the transposition sense the problem
$$
\left\{\begin{aligned}
\partial_tw -\Delta w+A\cdot\nabla w+\mu w &= H, \quad \mbox{in }\Omega\times(0,T],\\
 \partial_\nu w&= 0, \quad \mbox{on } \partial\Omega\times(0,T), \\
w&=0, \quad\mbox{in } \Omega\times \{0\},
\end{aligned}\right.
$$
with
$$H(x,t)=h(t)\sum_{k=1}^N\lambda_k(t)\delta_{x_k}(x)+h'(t)u(x,t),\quad (x,t)\in \Omega\times(0,T].$$
By condition \eqref{lambda} on the amplitude $\lambda_k$, we deduce that $h\lambda_k\equiv0$ on $[0,T]$ for $k=1,\ldots,N$, which implies $H(x,t)=h'(t)u(x,t)$, for $(x,t)\in \Omega\times(0,T)$, and $H\in L^2(0,T;L^2(\Omega))$. Thus, by \cite[Chapter 4, Theorem 6.1]{LM2}, we deduce that $w\in H^1(0,T;L^2(\Omega))\cap L^2(0,T;H^2(\Omega))$. Now, fix $h_1\in C^\infty([0,T];[0,1])$ such that,  $h_1=0$ on $[0,T_0+4\delta]$ and $h_1=1$ on $[T_1-\delta,T]$ and fix $w_1(x,t)=h_1(t)u(x,t)$, $(x,t)\in \Omega\times(0,T]$. One can easily check that $h_1h=h_1$ and $w_1(x,t)=h_1(t)u(x,t)=h_1(t)w(x,t)$, for $(x,t)\in \Omega\times(0,T)$. Moreover, by noting $w\in H^1(0,T;L^2(\Omega))\cap L^2(0,T;H^2(\Omega))$, we deduce that $w_1\in H^1(0,T;L^2(\Omega))\cap L^2(0,T;H^2(\Omega))$.
Then, by repeating the preceding argument, we deduce that $w_1$ solves
$$
\left\{\begin{aligned}
\partial_tw_1 -\Delta w_1+A\cdot\nabla w_1+\mu w_1 &= H_1, \quad \mbox{in }\Omega\times(0,T],\\
 \partial_\nu w_1&= 0,\quad \mbox{on } \partial\Omega\times(0,T), \\
w_1&=0, \quad \mbox{in } \Omega\times \{0\},
\end{aligned}\right.$$
with $$H_1(x,t)=\underbrace{h_1(t)h'(t)}_{=0}u(x,t)+h_1'(t)w(x,t)=h_1'(t)w(x,t),\quad (x,t)\in\Omega\times(0,T).$$
Since the function $H_1\in H^1(0,T;L^2(\Omega))\cap L^2(0,T;H^2(\Omega))$ satisfies the condition $H_1(x,0)=0$, $x\in\Omega$, one can check that
$w_2=\partial_tw_1$ solves the problem
$$
\left\{\begin{aligned}
\partial_tw_2 -\Delta w_2+A\cdot\nabla w_2+\mu w_2 &= \partial_tH_1, \quad \mbox{in }\Omega\times(0,T],\\
 \partial_\nu w_2&= 0,\quad \mbox{on } \partial\Omega\times(0,T), \\
w_2&=0, \quad \mbox{in } \Omega\times \{0\},
\end{aligned}\right.$$
and, by \cite[Chapter 4, Theorem 6.1]{LM2} again, we deduce that
$\partial_tw_1=w_2\in H^1(0,T;L^2(\Omega))\cap L^2(0,T;H^2(\Omega))$. Therefore, we have   $w_1\in H^1(0,T;H^2(\Omega))$,  which implies $u|_{\Omega\times [T_1,T]}=w_1|_{\Omega\times [T_1,T]}\in H^1(T_1,T;H^2(\Omega))$. By combining the  regularity properties of Steps 1 and 2, we obtain the assertion \eqref{t4aa}.

\ \\
\noindent\textbf{Step 3.} At this step, we show the estimate \eqref{t4a}. By fixing $\epsilon_0=(T-T_1)/8$ and applying the Carleman estimate in Proposition \ref{p1}, we get
\bel{t4b} \norm{u}_{H^1(T_1+2\epsilon_0,T-2\epsilon_0;L^2(\Omega))}+\norm{u}_{L^2(T_1+2\epsilon_0,T-2\epsilon_0;H^2(\Omega))}\leq C \norm{u}_{H^1(S\times (T_1,T))}.\ee
Fix $h_2\in C^\infty([0,T];[0,1])$ such that, $h_2=0$ on $[0,T_1+3\epsilon_0]$ and $h_2=1$ on $[T-3\epsilon_0,T]$ and let $u_1(x,t)=h_2(t)u(x,t)$, $(x,t)\in \Omega\times(0,T]$. By repeating the argument in Step 2, we deduce that $u_1\in H^1(0,T;H^2(\Omega))$ solves
$$
\left\{\begin{aligned}
\partial_tu_1 -\Delta u_1+A\cdot\nabla u_1+\mu u_1&=h_2'(t)u(x,t), \quad \mbox{in }\Omega\times(0,T],\\
 \partial_\nu u_1&= 0,\quad \mbox{on } \partial\Omega\times(0,T), \\
u_1&=0, \quad \mbox{in } \Omega\times \{0\}.
\end{aligned}\right.
$$
From \cite[Chapter 4, Theorem 6.1]{LM2} and \eqref{t4b}, we get
$$\begin{aligned}
&\norm{u(\cdot,T)}_{H^2(\Omega)}=\norm{u_1(T,\cdot)}_{H^2(\Omega)}\\
\leq& C\norm{u_1}_{H^1(0,T;H^2(\Omega))}
\leq C\norm{h_2'(t)u}_{H^1(0,T;L^2(\Omega))}\\
\leq &C\norm{u}_{H^1(T_1+2\epsilon_0,T-2\epsilon_0;L^2(\Omega))}\leq C \norm{u}_{H^1(S\times (T_1,T)}.\end{aligned}$$
From this last estimate, we deduce the estimate \eqref{t4a}.
\end{proof}

In view of Theorem \ref{t4}, we next consider a suitable time extension of the solution $u$ of problem \eqref{eq1} to $\Omega\times(0,+\infty)$ as well as its Laplace transform in time. More precisely, let $v_T$ be the solution of the following initial boundary value problem
\begin{equation}\label{eq2}
\left\{\begin{aligned}
\partial_tv -\Delta v+A\cdot\nabla v+\mu v &=  0,\quad \mbox{in }\Omega\times(T,+\infty),\\
 \partial_\nu v&= 0, \quad \mbox{on } \partial\Omega\times(T,+\infty), \\
v&=u_T, \quad \mbox{in } \Omega\times \{T\},
\end{aligned}\right.
\end{equation}
with $u_T=u(\cdot,T)$.
Since $u(\cdot,T)\in H^2(\Omega)$ satisfies $\partial_\nu u(\cdot,T)|_{\partial\Omega}\equiv 0$ and by \cite[Chapter 4, Theorem 6.1]{LM2}, we deduce that $v_T\in H^1_{loc}([T,+\infty);L^2(\Omega)) \cap L^2_{loc}([T,+\infty);H^2(\Omega))$. Moreover,
since the operator $\mathcal{L}=-\Delta +A\cdot\nabla +\mu$ with its domain $D(\mathcal{L})=\{f\in H^2(\Omega):\ \partial_\nu f|_{\partial\Omega}=0\}$ is sectorial \cite[Theorem 2.1]{A} (see also \cite[Theorem 2.5.1]{LLMP}), we deduce that there exists $D>|\mu|$ depending on $\Omega$, $A$ and $\mu$ such that for all $u_T\in \{v\in H^2(\Omega):\ \partial_\nu  u_T|_{\partial\Omega}=0\}$, we have
\begin{align}
(x,t)\mapsto e^{-Dt}v(x,t)&\in L^2([T,+\infty);H^2(\Omega))\cap H^1([T,+\infty);L^2(\Omega)),\nonumber\\
\label{est1}\norm{e^{-Dt}v}_{L^2([T,+\infty);H^2(\Omega))}&+\norm{e^{-Dt}v}_{H^1([T,+\infty);L^2(\Omega))}\leq C\norm{u_T}_{H^2(\Omega)},
\end{align}
with $C>0$ depending only on $T>0$, $\Omega$, $A$ and $\mu$. Fix now $u_\star$ defined on $\Omega\times(0,+\infty)$ by
$u_\star=u$ on $\Omega\times(0,T]$ and $u_\star=v_T$ on $\Omega\times(T,+\infty)$. For all $z\in\mathbb C$, consider the following boundary value problem
\begin{equation}\label{eq3}
\left\{\begin{aligned}
 -\Delta w_z+A\cdot\nabla w_z+\mu w_z +zw&=  u_0+\sum_{k=1}^N\widehat{\lambda}_k(z)\delta_{x_k}, \quad \mbox{in }\Omega,\\
 \partial_\nu w_z&= 0, \quad \mbox{on } \partial\Omega,
\end{aligned}\right.
\end{equation}
where $\widehat{\lambda}_k(z)$, $k=1,\ldots,N$, denotes the Laplace transform in time of $\lambda_k$ extended by zero given by
$\widehat{\lambda}_k(z)=\int_0^T\lambda_k(t)e^{-zt}\d t.$
Note that, by enlarging the constant $D>|\mu|$ of the estimate \eqref{est1},  we may assume without loss of generality that, for all $z\in\mathbb C$ satisfying $\re(z)>D$, problem \eqref{eq3} admits a unique solution in the transposition sense $w_z\in L^2(\Omega)$. Using these properties, we prove that the Laplace transform in time of the extension $u_\star$ of $u$ at $z$ is given by $w_z$.

\begin{prop}\label{p2} For all $z\in\mathbb C$ satisfying $\re(z)>D$, the Laplace transform in time $\widehat{u}_\star$ of the extension $u_\star$ is well defined at $z$ and there holds
\bel{p2a} \widehat{u}_\star(\cdot,z)=\int_0^{T}e^{-zt}u(\cdot,t)\d t+\int_{T}^{+\infty}e^{-zt}v_T(\cdot,t)\d t=w_z,\ee
where $w_z\in L^2(\Omega)$ is the unique solution in the transposition sense of problem \eqref{eq3}.
\end{prop}
\begin{proof} Without loss of generality, we may assume that $u_0\equiv0$. Fix $z\in\mathbb C$ satisfying $\re(z)>D$. In view of the estimate \eqref{est1}, the map $t\mapsto e^{-zt}u_\star(\cdot,t)$ belongs to $ L^1(0,+\infty;L^2(\Omega))$ with
\bel{p2b} \widehat{u}_\star(\cdot,z)=\int_0^{T}e^{-zt}u(\cdot,t)\d t+\int_{T}^{+\infty}e^{-zt}v_T(\cdot,t)\d t:=f_z+g_z.
\ee
By the governing equation in \eqref{eq2}, the estimate \eqref{est1} and integration by parts, we deduce that $g_z\in H^2(\Omega)$ solves
\begin{equation}\label{p2c}
\left\{\begin{aligned}
 -\Delta g_z+A\cdot\nabla g_z+\mu g_z +zg_z&=  u(\cdot,T)e^{-zT}, \quad \mbox{in }\Omega,\\
 \partial_\nu g_z&= 0, \quad \mbox{on } \partial\Omega,
\end{aligned}\right.
\end{equation}
Now fix $\delta\in(0,(T-T_1)/4)$ and consider $h\in C^\infty([0,T])$ satisfying $h=1$ on $[0,T_1+\delta]$ and $h=0$ on $[T-\delta,T]$. Note that
$$f_z=\int_0^{T}e^{-zt}h(t)u(\cdot,t)\d t+\int_0^{T}e^{-zt}(1-h(t))u(\cdot,t)\d t:=f_z^1+f_z^2.$$
From Theorem \ref{t4}, we have $t\mapsto e^{-zt}(1-h(t))u(\cdot,t)\in H^1(0,T;H^2(\Omega))$ and, by repeating the preceding argumentation, we deduce that $f_z^2\in H^2(\Omega)$ solves
\begin{equation}\label{p2d}
\left\{\begin{aligned}
 -\Delta f_z^2+A\cdot\nabla f_z^2+\mu f_z^2 +zf_z^2&=-  \int_0^{T}e^{-zt}h'(t)u(\cdot,t)\d t-u(\cdot,T)e^{-zT}, \quad \mbox{in }\Omega,\\
 \partial_\nu f_z^2&= 0, \quad \mbox{on } \partial\Omega.
\end{aligned}\right.
\end{equation}
By fixing $\phi\in \{w\in H^2(\Omega):(\partial_\nu w+(A\cdot\nu)w)|_{\partial\Omega}\equiv0\}$ and noting that $u$ solves problem \eqref{eq1} in the transposition sense and applying the condition \eqref{lambda}, we obtain
\begin{align*}
&\int_\Omega f_z^1( -\Delta \phi-\nabla\cdot(A\phi)+\mu \phi +z\phi)\d x\\
=&\int_0^T\!\!\int_{\Omega}u(-\partial_t-\Delta -\nabla\cdot(A\cdot)+\mu  )h(t)e^{-zt}\phi(x)\d x\d t+\int_\Omega \left(\int_0^{T}e^{-zt}h'(t)u(\cdot,t)\d t\right)\phi \d x\\
=&\sum_{k=1}^N\left(\int_0^{T_0}\lambda_k(t)e^{-zt}\underbrace{h(t)}_{=1}\d t\right)\phi(x_k)+\int_\Omega \left(\int_0^{T}e^{-zt}h'(t)u(\cdot,t)\d t\right)\phi\d x\\
=& \sum_{k=1}^N\widehat{\lambda_k}(z)\phi(x_k)+\int_\Omega \left(\int_0^{T}e^{-zt}h'(t)u(\cdot,t)\d t\right)\phi\d x.
\end{align*}
By combining this identity with \eqref{p2c}-\eqref{p2d}, we deduce that $\widehat{u_\star}(\cdot,z)\in L^2(\Omega)$ solves problem \eqref{eq3} in the transposition sense and the uniqueness of solutions for the problem implies  $\widehat{u_\star}(\cdot,z)=w_z$.
\end{proof}

The following  estimate will be crucial in the analysis below.
\begin{prop}\label{p3} Let $\tau\in\R$, $s>D$ and consider $z=s+i\tau$. For all $v_z\in H^2(\Omega)$ satisfying
$$-\Delta v_z-A\cdot\nabla v_z+\mu v_z +zv_z=0\quad \mbox{in }\Omega,$$
there exists $C>0$ depending on $\Omega$, $T$, $A$ and $\mu$ such that
\begin{align}\label{p3a}
&\abs{\sum_{k=1}^N\widehat{\lambda}_k(s+i\tau)v_z(x_k)+\int_\Omega u_0v_z\d x}\\
\leq &C\norm{\partial_\nu v_z+(A\cdot\nu)v_z}_{L^2(\partial\Omega)}(\norm{u}_{H^1(S\times(T_1,T))}+\norm{u}_{L^2(\partial\Omega\times(0,T))}).\nonumber
\end{align}
\end{prop}
\begin{proof}
Let ${\rm I}=\sum_{k=1}^N\widehat{\lambda}_k(s+i\tau)v_z(x_k)+\int_\Omega u_0v_z\d x$. By Theorem \ref{t4} and Proposition \ref{p2}, for all $r\in(0,r_0)$, the solution $w_z$ of problem \eqref{eq3} restricted to $\Omega_r$ belongs to $H^2(\Omega_r)$, and we have
\begin{equation}\label{p3ab}-\Delta w_z(x)+A\cdot\nabla w_z(x)+\mu w_z(x) +zw_z(x)=u_0(x),\quad x\in\Omega_r.\end{equation}
Now fix $\chi\in C^\infty(\overline{\Omega};[0,1])$ such that $\chi=1$ on $\overline{\Omega_{r_0/2}}$ and $\chi=0$ on $\Omega\setminus \Omega_{r_0/4}$. Then we have
\begin{align*}
{\rm I}=&\bigg(\sum_{k=1}^N\widehat{\lambda}_k(s+i\tau)(1-\chi(x_k))v_z(x_k)+\int_\Omega u_0(1-\chi) v_z\d x\bigg)+\int_\Omega u_0\chi v_z\d x=:{\rm II} + {\rm III}.
\end{align*}
By the choice of $\chi$, $(1-\chi)v_z$ belongs to $H^2(\Omega)$ and satisfies $(\partial_\nu +(A\cdot \nu))(1-\chi)v_z=0$ on $\partial\Omega$. Thus, by the definition of the solution in the sense of transposition for $w_z$, we have
\begin{align*}
    {\rm II} &=\int_\Omega w_z(-\Delta ((1-\chi)v_z)-A\cdot\nabla((1-\chi) v_z)+(\mu  +z)(1-\chi)v_z)\d x.
\end{align*}
By the product rule, the governing equation for $v_z$ and the condition $\chi=0$ on $\Omega\setminus\Omega_{r_0/4}$, we have
\begin{align*}
 {\rm II}   & = \int_{\Omega} w_z(2\nabla\chi\cdot\nabla v_z +(A\cdot\nabla\chi+\Delta\chi)v_z)\d x
    =\int_{\Omega_{r_0/4}} w_z(2\nabla\chi\cdot\nabla v_z +(A\cdot\nabla\chi+\Delta\chi)v_z)\d x.
\end{align*}
Next by the choice of $\chi$ and the identity \eqref{p3ab}, we have
\begin{align*}
    {\rm III} = \int_{\Omega_{r_0/4}} (-\Delta w_z+A\cdot\nabla w_z+\mu w_z +zw_z)\chi v_z\d x
\end{align*}
Combining the last two identities, integrating by parts and using the boundary conditions of $w_z$ and $\chi$ lead to
\begin{align*}
    {\rm I} &= \int_{\Omega_{r_0/4}} w_z(2\nabla\chi\cdot\nabla v_z +(A\cdot\nabla\chi+\Delta\chi)v_z)\d x+\int_{\Omega_{r_0/4}} (-\Delta w_z+A\cdot\nabla w_z+\mu w_z +zw_z)\chi v_z\d x\\
&=\int_{\partial\Omega} w_z (\partial_\nu(\chi v_z)+(A\cdot\nu)\chi v_z)\d\sigma +\int_{\Omega}(w_z\chi) (-\Delta v_z-A\cdot\nabla v_z+\mu v_z +zv_z)\d x.
\end{align*}
By the condition $\chi=1$ on a neigborhood of $\partial\Omega$ and the governing equation for $v_z$, we get
\begin{align*}
   |{\rm I}| &=\abs{\int_{\partial\Omega} w_z(\partial_\nu v_z+(A\cdot\nu)v_z)\d\sigma(x)}
\leq \norm{\partial_\nu v_z+(A\cdot\nu)v_z}_{L^2(\partial\Omega)}\norm{w_z}_{L^2(\partial\Omega)}.
\end{align*}
By Proposition \ref{p2} and \eqref{p2b}, we deduce
$$\norm{w_z}_{L^2(\partial\Omega)}\leq \norm{f_z}_{L^2(\partial\Omega)}+\norm{g_z}_{L^2(\partial\Omega)}.$$
From the estimate \eqref{est1} and the Cauchy-Schwarz inequality, since $s>D$, we get
\begin{align*} \norm{g_z}_{L^2(\partial\Omega)}&\leq \int_T^{+\infty}e^{-(s-D)t}\norm{e^{-Dt}v_T(\cdot,t)}_{L^2(\partial\Omega)}\d t\\
&\leq C\norm{e^{-Dt}v_T}_{L^2([T,+\infty);H^2(\Omega))}\leq C\norm{u(\cdot,T)}_{H^2(\Omega)},
\end{align*}
with $C>0$ depending only on $\Omega$, $T$, $A$ and $\mu$. Then, the estimate \eqref{t4a} implies
$$\norm{g_z}_{L^2(\partial\Omega)} \leq C\norm{u}_{H^1(S\times(T_1,T))}.$$
In the same way, we have
$$\norm{f_z}_{L^2(\partial\Omega)}\leq \int_0^T\norm{u(\cdot,t)}_{L^2(\partial\Omega)}dt\leq \sqrt{T}\norm{u}_{L^2(0,T;L^2(\partial\Omega))}.$$
By combining these two estimates, we obtain \eqref{p3a}.\end{proof}

\section{Proofs of the main results}\label{sec:proof}
In this section, we present the proofs of the main results in Section \ref{sec:main}.
\subsection{Proof of Theorem \ref{t2}}
We divide the lengthy proof into three steps.

\noindent\textbf{Step 1}: In this step, we will show that there exists $r_1>0$ depending on $\Omega$, $s$, $A$, $\mu$, $T$, $r_\star$ and $M$ such that
\begin{equation}\label{t2a}
\sup_{\tau\in\R}|\widehat{\lambda^1}((D+1)+i\tau)|\geq r_1.
\end{equation}
To this end, we extend the map
$\lambda^1$ by zero to $\R$, still denoted by $\lambda^1$. In light of \cite[Theorem 11.4, Chapter 1]{LM1}, we deduce that $\lambda^1\in H^s(\R)$ and  there exists $C$ depending only on $T$ and $s$ such that $$\norm{\lambda^1}_{H^s(\R)}\leq C\norm{\lambda^1}_{H^s(0,T)}.$$
By combining this estimate with condition \eqref{t2b}, we derive
\bel{t2ac}\norm{e^{-(D+1)t}\lambda^1}_{H^s(\R)}\leq C\norm{\lambda^1}_{H^s(\R)}\leq CM,\ee
with $C>0$ depending on $A$, $\mu$, $\Omega$, $s$ and $T$.
By condition \eqref{t2ab} and Fourier-Plancherel formula, for all $R>0$, we have
\begin{align*}
&r_\star^2\leq\norm{\lambda^1}_{L^2(0,T)}^2
\leq e^{2(D+1)T}\norm{e^{-(D+1)t}\lambda^1}_{L^2(0,T)}^2\\
=& e^{2(D+1)T}\norm{e^{-(D+1)t}\lambda^1}_{L^2(\R)}^2\leq e^{2(D+1)T}(2\pi)^{-1}\norm{\widehat{e^{-(D+1)t}\lambda^1}(i\cdot)}_{L^2(\R)}^2\\
\leq &e^{2(D+1)T}(2\pi)^{-1}\left(2R\left(\sup_{\tau\in\R}\abs{\widehat{\lambda^1}((D+1)+ i\tau)}\right)^2+\int_{|\tau|>R}\abs{\widehat{e^{-(D+1)t}\lambda^1}(i\tau)}^2\d\tau\right).
 \end{align*}
Combining this inequality with the estimate \eqref{t2ac} gives
\begin{align*}
&\int_{|\tau|>R}\abs{\widehat{e^{-(D+1)t}\lambda^1}(i\tau)}^2\d\tau\leq \int_{|\tau|>R}\left(\frac{(1+|\tau|)^{2s}}{(1+R)^{2s}}\right)\abs{\widehat{e^{-(D+1)t}\lambda^1}(i\tau)}^2\d\tau\\
\leq& R^{-2s}\int_{\R}(1+|\tau|)^{2s}\abs{\widehat{e^{-(D+1)t}\lambda^1}(i\tau)}^2\d\tau= R^{-2s}\|{e^{-(D+1)t}\lambda^1}\|_{H^s(\R)}^2\leq R^{-2s}C^2M^2.
\end{align*}
Therefore, we have
$$   R^{-2s}C^2M^2  +2R\left(\sup_{\tau\in\R}\abs{\widehat{\lambda^1}((D+1)+ i\tau)}\right)^2    \geq e^{-2(D+1)T}(2\pi)r_\star^2,\quad \forall R>0.$$
Since $\lambda^1\not\equiv0$, by the injectivity of Laplace transform, we deduce
\begin{align*}
\underset{\tau\in\R}{\sup}\abs{\widehat{\lambda^1}((D+1)+ i\tau)}> 0.
\end{align*}
Then, by choosing $R=\left(\underset{\tau\in\R}{\sup}\abs{\widehat{\lambda^1}((D+1)+ i\tau)}\right)^{-\frac{2}{2s+1}}$, we get
$$\left(\sup_{\tau\in\R}\abs{\widehat{\lambda^1}((D+1)+ i\tau)}\right)^{\frac{4s}{2s+1}}(2+C^2M^2)   \geq e^{-2(D+1)T}2\pi r_\star^2.$$
 Therefore, by fixing
$$r_1=\left(\frac{ e^{-2(D+1)T}2\pi r_\star^2}{2+C^2M^2}\right)^{\frac{2s+1}{4s}},$$
we obtain the desired estimate \eqref{t2a}.\\

\noindent\textbf{Step 2}: In this step, we will prove the estimate \eqref{t2ab1}.
Let $u=u^1-u^2$. Then $u$ solves problem \eqref{eq1} with $u_0\equiv0$, $N=2$ and, for $j=1,2$,  $\lambda_j=(-1)^{j+1}\lambda^j$, $x_j=x^j$.
Let $\tau\in\R$ and fix $p=D+1+i\tau$. By Theorem \ref{t4} and Proposition \ref{p3}, for all $v_p\in H^2(\Omega)$ satisfying
$$-\Delta v_p-A\cdot\nabla v_p+\mu v_p +pv_p=0\quad \mbox{in } \Omega,$$
we derive
\begin{align}\label{t2c}
&\abs{\widehat{\lambda^1}((D+1)+i\tau)v_p(x^1)-\widehat{\lambda^2}((D+1)+i\tau)v_p(x^2)}\\
\leq &C\norm{(\partial_\nu v_p+(A\cdot\nu)v_p)}_{L^2(\partial\Omega)}(\norm{u}_{H^1(S\times(T_1,T))}+\norm{u}_{L^2(0,T;L^2(\partial\Omega))}).\nonumber
\end{align}
Let $\Sigma_3$ be the set of permutations of the set $\{1,2,3\}$ and, for all $x=(x_1,x_2,x_3)\in\R^3$ and $\sigma\in\Sigma_3$, let $z_\sigma=x_{\sigma(1)}+ix_{\sigma(2)}$  and $z_\sigma^j=x_{\sigma(1)}^j+ix_{\sigma(2)}^j$, $j=1,2$, with $x^j=(x_1^j,x_2^j,x_3^j)$. One can check that, for any $\sigma\in\Sigma_3$, the function
$$w_\sigma(x)=e^{-\frac{A\cdot (x-x^1)}{2}}(z_\sigma-z_\sigma^2)\exp\left(\left(\mu +\frac{|A|^2}{4}+p\right)^{\frac 1 2}(x_{\sigma(3)}-x_{\sigma(3)}^1)\right),\quad x\in\Omega,$$
satisfies
\begin{align*}
-\Delta w_\sigma-A\cdot\nabla w_\sigma+\mu w_\sigma +pw_\sigma=0\quad \mbox{in }\Omega \quad\mbox{and}\quad w_\sigma(x^1)=z_\sigma^1-z_\sigma^2,\quad w_\sigma(x^2)=0.
\end{align*}
Therefore, by the estimate \eqref{t2c} with $v_p=w_\sigma$, we obtain
\begin{align}\label{t2d}
&\abs{\widehat{\lambda^1}((D+1)+i\tau)}|z_\sigma^1-z_\sigma^2|\\
\leq &C\norm{\partial_\nu w_\sigma+(A\cdot\nu)w_\sigma}_{L^2(\partial\Omega)}(\norm{u}_{H^1(S\times(T_1,T))}+\norm{u}_{L^2(0,T;L^2(\partial\Omega))}),\quad \sigma\in\Sigma_3.\nonumber
\end{align}
In view of the estimate \eqref{t2a}, we can find $\tau_0\in\R$ such that $|\widehat{\lambda^1}((D+1)+i\tau_0)|\geq \frac{r_1}{2}$. Thus, by fixing $\tau=\tau_0$ in \eqref{t2d}, we get
$$\begin{aligned}|z_\sigma^2-z_\sigma^1|&\leq 2r_1^{-1}C\norm{\partial_\nu w_\sigma+(A\cdot\nu)w_\sigma}_{L^2(\partial\Omega)}(\norm{u}_{H^1(S\times(T_1,T))}+\norm{u}_{L^2(\partial\Omega\times(0,T))})\\
&\leq C (\norm{u}_{H^1(S\times(T_1,T))}+\norm{u}_{L^2(0,T;L^2(\partial\Omega))}),\quad \sigma\in\Sigma_3,\end{aligned}$$
with the constant $C>0$ depending on $r_\star$, $s$, $\Omega$, $T_1$,  $S$ $A$, $\mu$, $M$ and $T$. This clearly implies the estimate \eqref{t2ab1}.\\

\noindent \textbf{Step 3}: We will complete the proof of the theorem by proving the estimate \eqref{t2ab2}. Fix $R>1$, $\tau\in[-R,R]$, $p=D+1+ i\tau$ and $\omega\in\mathbb S^{2}$ such that $\omega\cdot (x^1-x^2)=0$. Therefore, by choosing
$$v_{p}(x)=e^{-\frac{A\cdot (x-x^2)}{2}}\exp\left(\left(\mu +\frac{|A|^2}{4}+p\right)^{\frac 1 2}\omega\cdot (x-x^2)\right),\quad x\in\Omega$$
in \eqref{t2c}, we deduce
\begin{align}\label{t2e}
&\abs{\widehat{\lambda^1}((D+1)+ i\tau)e^{-\frac{A\cdot (x^1-x^2)}{2}}-\widehat{\lambda^2}((D+1)+ i\tau)}\\
\leq& C\norm{\partial_\nu v_{p}+(A\cdot\nu)v_p}_{L^2(\partial\Omega)}(\norm{u}_{H^1(S\times(T_1,T))}+\norm{u}_{L^2(0,T;L^2(\partial\Omega))}).\nonumber
\end{align}
Moreover, by the triangle inequality, we have
\begin{align*}
&\abs{\widehat{\lambda^1}((D+1)+ i\tau)e^{-\frac{A\cdot (x^1-x^2)}{2}}-\widehat{\lambda^2}((D+1)+ i\tau)}\\
\geq& \abs{\widehat{\lambda^1}((D+1)+ i\tau)-\widehat{\lambda^2}((D+1)+ i\tau)}-\abs{\widehat{\lambda^1}((D+1)+ i\tau)\left(e^{-\frac{A\cdot (x^1-x^2)}{2}}-1\right)}\\
\geq& \abs{\widehat{\lambda^1}((D+1)+ i\tau)-\widehat{\lambda^2}((D+1)+ i\tau)}-|A||x^2-x^1|\norm{\lambda^1}_{L^1(0,T)},
\end{align*}
and from the estimate \eqref{t2ab1}, we obtain
$$\begin{aligned}&\abs{\widehat{\lambda^1}((D+1)+ i\tau)e^{-\frac{A\cdot (x^1-x^2)}{2}}-\widehat{\lambda^2}((D+1)+ i\tau)}\\
\geq &\abs{\widehat{\lambda^1}((D+1)+ i\tau)-\widehat{\lambda^2}((D+1)+ i\tau)}-C(\norm{u}_{H^1(S\times(T_1,T))}+\norm{u}_{L^2(0,T;L^2(\partial\Omega))}),\end{aligned}$$
with the constant $C>0$ depending only on $r_\star$, $s$, $\Omega$, $A$, $T_1$, $S$,  $\mu$, $M$ and $T$. By combining this estimate with \eqref{t2e}, for all $\tau\in[-R,R]$, we get
\begin{equation}\label{t2f}
\abs{\widehat{\lambda^1}((D+1)+ i\tau)-\widehat{\lambda^2}((D+1)+ i\tau)}\leq C(1+\norm{ v_{p}}_{H^2(\Omega)})(\norm{u}_{H^1(S\times(T_1,T))}+\norm{u}_{L^2(0,T;L^2(\partial\Omega))}),
\end{equation}
with $C>0$ depending only on $r_\star$, $s$, $\Omega$, $M$, $T_1$, $S$, $A$, $\mu$ and $T$. Meanwhile, we have
$$\norm{ v_{p}}_{H^2(\Omega)}\leq C\left(|\mu| +\frac{|A|^2}{4}+D+1+R\right) \exp\left(c\left(|\mu| +\frac{|A|^2}{4}+D+1+R\right)^{\frac 1 2}\right),\quad x\in\Omega,$$
where $C>0$ depends only on  $\Omega$, $A$ and $\mu$, and $c>0$ depends only on $\Omega$. Therefore,  from the estimate \eqref{t2f}, for $R>|\mu| +\frac{|A|^2}{4}+D+1$, we find
\begin{equation}\label{t2g}
\sup_{\tau\in[-R,R]}\abs{\widehat{\lambda^1}((D+1)+ i\tau)-\widehat{\lambda^2}((D+1)+ i\tau)}\leq CRe^{cR^{\frac 1 2}}(\norm{u}_{H^1(S\times(T_1,T))}+\norm{u}_{L^2(L^2(0,T;L^2(\partial\Omega))}),
\end{equation}
where $C,c>0$ depend only on  $\Omega$, $A$, $T_1$, $\mu$, $r_\star$, $M$, $S$ and $T$.
In addition, by extending $\lambda=\lambda^1_1-\lambda^2_1$ by zero to $\R$ and applying the Fourier-Plancherel formula, we get
\begin{align*}
&\norm{\lambda^1-\lambda^2}_{L^2(0,T)}^2\leq e^{2(D+1)T}\norm{e^{-(D+1)t}\lambda^1-e^{-(D+1)t}\lambda^2}_{L^2(0,T)}^2\\
\leq &e^{2(D+1)T}\norm{e^{-(D+1)t}\lambda}_{L^2(\R)}^2\leq e^{2(D+1)T}(2\pi)^{-1}\norm{\widehat{e^{-(D+1)t}\lambda}(i\cdot)}_{L^2(\R)}^2\\
\leq & C\left(2R\left(\sup_{\tau\in[-R,R]}\abs{\widehat{\lambda^1}((D+1)+ i\tau)-\widehat{\lambda^2}((D+1)+ i\tau)}\right)^2+\int_{|\tau|>R}\abs{\widehat{e^{-(D+1)t}\lambda}(i\tau)}^2\d\tau\right).
 \end{align*}
Combining this estimate with \eqref{t2g} and fixing $\gamma=(\norm{u}_{H^1(S\times(T_1,T))}+\norm{u}_{L^2(\partial\Omega\times(0,T))})$ give
$$\begin{aligned}&\norm{\lambda^1-\lambda^2}_{L^2(0,T)}^2\leq C\left(R^3e^{2cR^{\frac 1 2}}\gamma^2+\int_{|\tau|>R}\abs{\widehat{e^{-(D+1)t}\lambda}(i\tau)}^2\d\tau\right).
 \end{aligned}$$
Meanwhile, by condition \eqref{t2b}, we obtain
\begin{align*}
&\int_{|\tau|>R}\abs{\widehat{e^{-(D+1)t}\lambda}(i\tau)}^2\d\tau\leq\int_{|\tau|>R}\left(\frac{1+|\tau|^2}{1+R^2}\right)\abs{\widehat{e^{-(D+1)t}\lambda}(i\tau)}^2\d\tau\\
\leq& R^{-2}\int_{\R}(1+|\tau|^2)\abs{\widehat{e^{-(D+1)t}\lambda}(i\tau)}^2\d\tau\leq R^{-2}\norm{e^{-(D+1)t}\lambda}_{H^1(\R)}^2
\leq CR^{-2}\norm{\lambda}_{H^1(\R)}^2\leq CM^2R^{-2},
\end{align*}
which implies
$$\norm{\lambda^1-\lambda^2}_{L^2(0,T)}^2\leq C(e^{3cR^{\frac 1 2}}\gamma^2+R^{-2}),$$
where $C>0$ depends on $r_\star$, $s$, $\Omega$, $A$, $T_1$, $S$, $\mu$, $M$ and $T$. From this last estimate and the classical optimization argument, we obtain the desired estimate \eqref{t2ab2}.

\subsection{Proof of Theorems \ref{t3} and \ref{t100}}

The proof of  Theorem \ref{t3} relies on  the following lemma.

\begin{lem}\label{l1} Let $d=1,2,3$ and let $u$ be the solution in the transposition sense of problem \eqref{eq1}. Then, for all $v\in H^1(0,T;L^2(\Omega))\cap L^2(0,T;H^2(\Omega))$, we have
\begin{align}\label{l1a}
&\int_0^T\!\!\int_\Omega (-\partial_tv-\Delta v-A\cdot\nabla v+\mu v)u\d x\d t\\
=&-\int_0^T\!\!\int_{\partial\Omega} (\partial_\nu v +(A\cdot\nu)v)ud\sigma(x)\d t-\int_\Omega v(x,T)u(x,T)\d x+\sum_{j=1}^N\int_0^T\lambda_j(t)v(x_j,t)\d t.\nonumber
\end{align}
\end{lem}
\begin{proof} Let $U\subset \mathbb{R}^d$ be an open set such that $\{x_1,\ldots,x_N\}\subset U$, $\overline{U}\subset \Omega$ and fix $\phi\in C^\infty_0(\Omega;[0,1])$ such that $\phi=1$ in a neighborhood of $\overline{U}$. Consider also  $h\in C^\infty([0,T])$ satisfying $h=1$ on $[0,T_1+\delta]$ and $h=0$ on $[T-\delta,T]$ with $\delta\in(0,(T-T_1)/4)$. We split the function $v\in H^1(0,T;L^2(\Omega))\cap L^2(0,T;H^2(\Omega))$ into four terms as
$$v(x,t)=\underbrace{h(t)\phi(x)v(x,t)}_{:=v_1(x,t)}+\underbrace{(1-h(t))\phi(x)v(x,t)}_{:=v_2(x,t)}+\underbrace{h(t)(1-\phi(x))v(x,t)}_{:=v_3(x,t)}+\underbrace{(1-h(t))(1-\phi(x))v(x,t)}_{:=v_4(x,t)}.$$
Then, we find
\begin{equation}\label{l1b}
\int_0^T\!\!\int_\Omega (-\partial_tv-\Delta v-A\cdot\nabla v+\mu v)u\d x\d t=\sum_{j=1}^4\int_0^T\!\!\int_\Omega (-\partial_tv_j-\Delta v_j-A\cdot\nabla v_j+\mu v_j)u\d x\d t.
\end{equation}
Since $v_1\in L^2(0,T;H^2(\Omega)\cap H^1_0(\Omega))$, $v_1(\cdot,T)\equiv0$ and $\partial_\nu v_1 + A\cdot \nu v_1=0$ on $\partial\Omega\times(0,T)$, by the definition of the solution in the transposition sense, we obtain
\bel{l1c} \int_0^T\!\!\int_\Omega (-\partial_tv_1-\Delta v_1-A\cdot\nabla v_1+\mu v_1)u\d x\d t=\sum_{j=1}^N\int_0^T\lambda_j(t)v_1(x_j,t)\d t=\sum_{j=1}^N\int_0^T\lambda_j(t)v(x_j,t)\d t.\ee
Moreover, in view of Theorem \ref{t2}, we have $u|_{(\Omega\setminus \overline{U})\times (0,T)}\in L^2\left(0,T; H^2(\Omega\setminus \overline{U})\right)\cap H^1(0,T;L^2(\Omega\setminus \overline{U}))$, $u||_{\Omega\times (T_1,T)}\in H^1(T_1,T;H^2(\Omega))$ and \eqref{t4aa} is fulfilled. Therefore, by integration by parts, we get
\begin{align*}
&\int_0^T\!\!\int_\Omega (-\partial_tv_2-\Delta v_2-A\cdot\nabla v_2+\mu v_2)u\d x\d t
=\int_\Omega\!\int_{T_1}^T (-\partial_tv_2-\Delta v_2-A\cdot\nabla v_2+\mu v_2)u\d x\d t\\
=&-\int_0^T\!\!\int_{\partial\Omega}u(\partial_\nu v_2+(A\cdot\nu)v_2)\d\sigma(x)\d t-\int_\Omega v_2(x,T)u(x,T)\d x
=-\int_\Omega \chi(x) v(x,T)u(x,T)\d x,\\
&\sum_{j=3}^4\int_0^T\!\!\int_\Omega (-\partial_tv_j-\Delta v_j-A\cdot\nabla v_j+\mu v_j)u\d x\d t
=\sum_{j=3}^4\int_{\Omega\setminus \overline{U}}\!\!\int_{0}^T(-\partial_tv_j-\Delta v_j-A\cdot\nabla v_j+\mu v_j)u\d x\d t\\
=&\sum_{j=3}^4-\int_0^T\!\!\int_{\partial\Omega }u(\partial_\nu v_j+(A\cdot\nu)v_j)\d\sigma(x)\d t-\int_\Omega v_j(x,T)u(x,T)\d x\\
=&-\int_0^T\!\!\int_{\partial\Omega}u(\partial_\nu v+(A\cdot\nu)v)\d\sigma(x)\d t-\int_\Omega (1-\chi(x)) v(x,T)u(x,T)\d x.
\end{align*}
Combining these two identities with the identities \eqref{l1b}-\eqref{l1c}, we get the desired assertion \eqref{l1a}.\end{proof}

We can now state the proof of Theorems \ref{t3} and \ref{t100}.

\noindent\textbf{Proof of Theorem \ref{t3}} Let $u=u^1-u^2$. Then the function $u$ solves problem \eqref{eq1} with $u_0\equiv0$, $N$ replaced by $2N$ and, for $j=1,2$,  $\lambda_{k+(j-1)N}=(-1)^{j+1}\lambda_k^j$, $x_{k+(j-1)N}=x_k^j$.  For every $x=(x_1,x_2)\in\R^2$, we associate the complex number $z=x_1+ix_2$ and, for $k=1,\ldots,N$ and $j=1,2$, we fix $z_k^j=(x_k^j)_1+i(x_k^j)_2$ with $x_k^j=((x_k^j)_1, (x_k^j)_2)$. Fix $\ell\in\{1,\ldots,N\}$ and, without loss of generality, assume that $x_\ell^1\neq x_\ell^2$. Let
$$H_\ell(x)=e^{-\frac{A\cdot (x-x^1_\ell)}{2}} P_\ell(x_1+ix_2),\quad \ell=1,\ldots,N,$$
where, for $\ell=1,\ldots,N$, $P_\ell$ is the following complex  polynomial
$$P_\ell (z)=\prod_{k\neq \ell}(z-z_k^1)\prod_{m=1}^N(z-z_k^2),\quad z\in\mathbb C.$$
Fix $\ell=1,\ldots,N$ and, for all $(x,t)\in \Omega\times(0,T]$, direct computation gives
$$\begin{aligned}(-\partial_tH_\ell-\Delta H_\ell -A\cdot\nabla H_\ell+\mu H_\ell)(x,t)&=(-\Delta H_\ell -A\cdot\nabla H_\ell+\mu H_\ell)(x)\\
&=-e^{-\frac{A\cdot x}{2}}\Delta P_\ell(x_1+ix_2)+\Big(\mu+\frac{|A|^2}{4}\Big) H_\ell(x).\end{aligned}$$
Since $P_\ell$ is holomorphic and $\mu+\frac{|A|^2}{4}=0$, we get
$$(-\partial_tH_\ell-\Delta H_\ell -A\cdot\nabla+\mu H_\ell)(x,t)=0,\quad (x,t)\in \Omega\times(0,T].$$
Therefore, by the identity \eqref{l1a} with $v(x,t)=H_\ell (x)$, $(x,t)\in \Omega\times(0,T]$, and $u=u_1-u_2$, we get
$$-\int_0^T\!\!\int_{\partial\Omega} u(\partial_\nu H_\ell +(A\cdot\nu)H_\ell) \d\sigma(x)\d t-\int_\Omega u(x,T)H_\ell(x)\d x+\sum_{j=1}^2(-1)^{j+1}\sum_{k=1}^N\left(\int_0^T\lambda_k^j(t)\d t\right)H_\ell(x_k^j)=0.$$
In addition,  $H_\ell(x_k^j)=0$ for $(k,j)\neq (\ell,1)$, which implies
$$\left(\int_0^T\lambda_\ell^1(t)\d t\right)H_\ell(x_\ell^1)=\int_0^T\!\int_{\partial\Omega} u(\partial_\nu H_\ell (x)+(A\cdot\nu)H_\ell(x))\d\sigma(x)\d t+\int_\Omega u(x,T)H_\ell(x)\d x.$$
Combining this with Theorem \ref{t4} and condition \eqref{t3a} yields
\bel{t3g}|H_\ell(x_\ell^1)|\leq C(\norm{u}_{H^1(S\times(T_1,T))}+\norm{u}_{L^2(0,T;L^2(\partial\Omega))}),\ee
with  $C>0$  depending  on $r_2$,  $\Omega$, $A$, $\mu$, $T_1$, $S$ and $T$. Moreover, by conditions \eqref{t3c}-\eqref{t3d}, we get
$$|H_\ell(x_\ell^1)|=\prod_{k\neq \ell}|x_\ell^1-x_k^1|\prod_{m=1}^N|x_\ell^1-x_k^2|\geq |x^1_\ell-x_{\sigma(\ell)}^2|^N\delta_0^{N-1}.$$
Combining this with the estimate \eqref{t3g} gives
$$|x^1_\ell-x_{\sigma(\ell)}^2|\leq C(\norm{u}_{H^1(S\times(T_1,T))}+\norm{u}_{L^2(0,T;L^2(\partial\Omega))})^{\frac{1}{N}},$$
with $C>0$  depending  on $r_2$, $\delta_0$, $\Omega$, $A$, $\mu$, $T_1$, $S$ and $T$. This last estimate clearly implies the H\"{o}lder stability estimate \eqref{t3e}.
This completes the proof of Theorem \ref{t3}.\qed
\ \\

\noindent\textbf{Proof of Theorem \ref{t100}} In a similar way to Theorem \ref{t2}, we can prove that the estimates \eqref{t2a} and \eqref{t2c} hold. For $j=1,2$, we set $x^j=(x^j_1,x^j_2)$, and we fix $\tau\in\R$ and $p=D+1+i\tau$. For $k=1,2$, let
$$w_k(x)=e^{-\frac{A\cdot (x-x^1)}{2}}(x_k-x_k^2)\exp\left(\left(\mu +\frac{|A|^2}{4}+p\right)^{\frac 1 2}(x_{3-k}-x_{3-k}^1)\right),\quad x=(x_1,x_2)\in\Omega$$
and note that
\begin{align*}
-\Delta w_k-A\cdot\nabla w_k+\mu w_k +pw_k=0\quad \mbox{in }\Omega \quad\mbox{and}\quad w_k(x^1)=x_k^1-x_k^2,\quad w_k(x^2)=0.
\end{align*}
Then by \eqref{t2c} with $v_p=w_k$, $k=1,2$, we get
\begin{align}
\abs{\widehat{\lambda^1}((D+1)+i\tau)}|x_k^1-x_k^2|\leq C\norm{\partial_\nu w_k+(A\cdot\nu)w_k}_{L^2(\partial\Omega)}(\norm{u}_{H^1(S\times(T_1,T))}+\norm{u}_{L^2(0,T;L^2(\partial\Omega))}).\nonumber
\end{align}
In view of condition \eqref{t2a}, we can find $\tau_0\in\R$ such that $|\widehat{\lambda^1}((D+1)+i\tau_0)|\geq \frac{r_1}{2}$. Thus, by fixing $\tau=\tau_0$, we get
$$\begin{aligned}|x_k^1-x_k^2|&\leq 2r_1^{-1}C\norm{\partial_\nu w_k+(A\cdot\nu)w_k}_{L^2(\partial\Omega)}(\norm{u}_{H^1(S\times(T_1,T))}+\norm{u}_{L^2(0,T;L^2(\partial\Omega))})\\
&\leq C (\norm{u}_{H^1(S\times(T_1,T))}+\norm{u}_{L^2(0,T;L^2(\partial\Omega))}),\quad k=1,2,\end{aligned}$$
with the constant $C>0$ depending on $r_\star$, $s$, $\Omega$, $T_1$,  $S$ $A$, $\mu$, $M$ and $T$. This clearly implies  \eqref{t2ab1}.

Now consider the estimate \eqref{t2ab2}. To this end, fix $R>1$, $\tau\in[-R,R]$, $p=D+1+ i\tau$, $\omega\in\mathbb S^{1}$ such that $\omega\cdot (x^1-x^2)=0$. By choosing
$$v_{p}(x)=e^{-\frac{A\cdot (x-x^2)}{2}}\exp\left(\left(\mu +\frac{|A|^2}{4}+p\right)^{\frac 1 2}\omega\cdot (x-x^2)\right),\quad x\in\Omega$$
in \eqref{t2d}, we get
\begin{align*}
&\abs{\widehat{\lambda^1}((D+1)+i\tau)v_p(x^1)-\widehat{\lambda^2}((D+1)+i\tau)v_p(x^2)}\\
\leq &C\norm{\partial_\nu v_p+(A\cdot\nu)v_p}_{L^2(\partial\Omega)}(\norm{u}_{H^1(S\times(T_1,T))}+\norm{u}_{L^2(0,T;L^2(\partial\Omega))}).
\end{align*}
Then, repeating the argument in the last step of Theorem \ref{t2} gives the assertion \eqref{t2ab2}.\qed
\subsection{Proof of Theorem \ref{t5}}
Lt $u=u^1-u^2$. Then $u$ solves problem \eqref{eq1} with $u_0\equiv0$, for $j=1,2$,  $\lambda_{j}=(-1)^{j+1}\lambda_1^j$, $x_{j}=x_1^j$.  We define the map
$$g(x)=e^{-\frac{A (x-x_1^1)}{2}} (x-x_1^2),\quad x\in (0,\ell).$$
One can easily check that
$$(-\partial_tg-\partial_x^2g -A\cdot \partial_xg+\mu g)(x,t)=-g''(x)-A\cdot g'(x)+\mu g(x)=0,\quad (x,t)\in \Omega\times(0,T].$$
Therefore, by the identity \eqref{l1a} with $v(x,t)=g(x)$, $(x,t)\in \Omega\times(0,T]$, and $u=u_1-u_2$, we get
$$\sum_{j=0}^1(-1)^j\left[(g'(j\ell)+Ag(j\ell))\int_0^Tu(j\ell,t)\d t\right]-\int_0^\ell u(x,T)g(x)\d x+\sum_{j=1}^2(-1)^{j+1}\left(\int_0^T\lambda_1^j(t)\d t\right)g(x_1^j)=0.$$
Then, we obtain
$$\abs{\int_0^T\lambda_1^1(t)\d t} |x_1^1-x_1^2|\leq C\left(\sum_{j=0}^1\norm{u(j\ell,\cdot)}_{L^2(0,T)}+\norm{u(\cdot,T)}_{L^2(0,\ell)}\right).$$
By combining this with Theorem \ref{t4} and \eqref{t3a}, we get
$$ |x_1^1-x_1^2|\leq r_2^{-1}C\left(\sum_{j=0}^1\norm{u(j\ell,\cdot)}_{L^2(0,T)}+\norm{u(\ell,\cdot)}_{H^1(T_1,T)}\right),$$
with $C>0$  depending  on $r_2$, $\delta_0$, $\Omega$, $A$, $\mu$ and $T$.
To derive the estimate \eqref{t5c}, fix $R>1$, $\tau\in[-R,R]$, $z=D+1+ i\tau$. By choosing
$$v_{z}(x)=e^{-\frac{A (x-x^1_1)}{2}}e^{z^{\frac 1 2} (x-x_1^1)},\quad x\in(0,\ell)$$
in \eqref{t2d}, we get
\begin{align*}
&\abs{\widehat{\lambda_1^1}((D+1)+i\tau)v_z(x_1^1)-\widehat{\lambda_1^2}((D+1)+i\tau)v_z(x_1^2)}\\
\leq &CR_z\Big(\sum_{j=0}^1\norm{u(j\ell,\cdot)}_{L^2(0,T)}+\norm{u(\ell,\cdot)}_{H^1(T_1,T)}\Big),
\end{align*}
with
$$R_z=1+(1+|z|^{\frac 1 2})|e^{-z^{\frac 1 2}x_1^1}|+(1+|z|^{\frac 1 2})|e^{z^{\frac 1 2}(\ell-x_1^1)}|.$$
Then, by repeating the argument in the last step of Theorem \ref{t2},  we obtain the estimate \eqref{t5c}.

\section{Numerical results and discussions}\label{sec:numer}

Now we present several numerical examples of recovering point sources in one- and two-dimensional cases using the Levenberg-Marquardt algorithm \cite{Levenberg:1944,Marquardt:1963}.
Specifically, to recover the location $x^*$ of the point source and its amplitude $\lambda^*$, we define a nonlinear operator $F:(x,\lambda)\in \Omega \times L^2(0,T) \rightarrow u|_{[0,T]\times\partial\Omega}\in L^2([0,T];H^1(\partial\Omega))$, where $u$ solves problem \eqref{eq1} with parameters $x$ and $\lambda$. Starting from an initial guess $(x^0,\lambda^0)$, the algorithm generates a sequence of iterates $(x^k,\lambda^k)_{k\geq0}$. At each iteration $k$, the update $(x^{k+1},\lambda^{k+1})$ is computed as
\begin{equation*}
(x^{k+1},\lambda^{k+1})=\operatorname{argmin}_{x,\lambda} J_k(x,\lambda),
\end{equation*}
with the functional $J_k(x,\lambda)$ based at $(x^k,\lambda^k)$ given by
\begin{align*}
    J_k(x,\lambda)=&\left\|F(x^{k},\lambda^k)-z^{\delta}+\partial_xF(x^k,\lambda^k)(x-x^k)+\partial_{\lambda} F(x^k,\lambda^k)(\lambda-\lambda^k) \right\|^2_{L^2([0,T]\times\partial\Omega)}\\
    & +\beta_x^k|x-x^k|^2+{\beta_{\lambda}^k}\|\lambda-\lambda^k\|_{L^2(0,T)}^2,
\end{align*}
where $z^\delta$ denotes the noisy data, and $\beta_x^k>0$ and $\beta_\lambda^k>0$ are regularization parameters for the location $x$ and amplitude $\lambda$, respectively. The terms $\partial_x F (x^k,\lambda^k)$ and $\partial_\lambda F (x^k,\lambda^k)$ denote the Jacobians of the mapping $F$ with respect to $x$ and $\lambda$, respectively. The parameters $\beta_x^k$ and $\beta_\lambda^k$ are updated geometrically with factors $\gamma_x,\gamma_\lambda\in(0,1)$ such that $\beta_x^{k+1}=\gamma_x\beta_x^k $ and $\beta_\lambda^{k+1}=\gamma_\lambda \beta_\lambda^k $. We use early stopping to avoid over-fitting during the iteration.

Throughout, the direct problems are solved using the Galerkin FEM with continuous piecewise linear elements in space and the backward Euler scheme in time \cite{Thomee:2006}. The exact data is generated on a finer spatial-temporal mesh than that used for solving the inverse problem in order to mitigate the inverse crime. The Jacobians $\partial_x F(x^k,\lambda^k)$ and $\partial_\lambda J(x^k,\lambda^k)$ are computed using a finite difference method and the adjoint method, respectively. The concrete hyperparameter settings, including the ones used in the Levenberg-Marquardt algorithm, are listed in Table~\ref{tab:mesh param}. The noisy data $z^\delta$ is generated by adding independent and identically distributed Gaussian noise of zero mean and standard deviation $\delta \norm{u_{x^*,\lambda^*}}_{L^\infty([0,T]\times\partial\Omega)}$ to the exact data $u_{x^*,\lambda^*}$, with $\delta$ denoting the relative noise level. We compute the mean and standard error of the $L^1$ errors in the recovered source locations and amplitudes over ten independent runs for each setting.

\begin{table}[hbt!]
    \centering
    \begin{threeparttable}
    \caption{Hyper-parameter settings for the experiments. In all examples, we take $\Delta t=\Delta x$, and denote by the subsripts $\rm d$ and $\rm i$ the parameters for direct and inverse problems, respectively. \label{tab:mesh param}}
    \begin{tabular}{ccccc}
    \toprule
    Example  &  $(\Delta t,\Delta x)_{\rm d}$ & $(\Delta t,\Delta x)_{\rm i}$ &  $(\beta_x^0,\beta_\lambda^0)$ & $(\gamma_x,\gamma_\lambda)$\\
    \midrule
\ref{ex:1}(i) &   $10^{-3}$  & $4\times 10^{-3}$ & (1, 5) & (0.8, 0.8) \\
\ref{ex:1}(ii) &  $10^{-3}$  & $4\times 10^{-3}$ & (1, 2) & (0.8, 0.8) \\
\ref{ex:2} & $5\times10^{-3}$ & $ 2\times10^{-2}$ & (1, 50) & (0.8, 0.8) \\
\ref{ex:3} & $5\times 10^{-3}$ & $2\times10^{-2}$ &(1, 50) & (0.8, 0.8) \\
\ref{ex:4} & $5\times10^{-3}$ & $2\times10^{-2}$ & (1, 50) & (0.8, 0.8) \\
    \bottomrule
    \end{tabular}
\end{threeparttable}
\end{table}

First, we investigate the one-dimensional case ($d=1$). The notation $\mathds{1}_{[0,1]}$ denotes the characteristic function of the interval $[0,1]$.
\begin{example}\label{ex:1}
$\Omega=(0,1)$, $T=2$, $A=0$, $\mu=1$, and $u_0\equiv0$. There is a single point source at $x=0.5$ with amplitude given by {\rm (i)} $\lambda(t)=0.5e^{-5t}$ and {\rm (ii)} $\lambda(t)=\mathds{1}_{[0,1]}(t)$.
\end{example}

Note that the amplitude $\lambda(t)$ in case (ii) belongs to the space $H^s(0,T)$ for any $s<1/2$. For case (i), the algorithm is initialized to $x^0=0.4$ and $\lambda^0(t)=0.8\lambda(t)$ for $t\in[0,2]$; for case (ii), we initialize with $x^0=0$ and $\lambda^0(t)=0$ for $t\in[0,2]$.
Fig.~\ref{fig:1} shows that the  reconstructed source location is very accurate in both cases (i) and (ii) for data with $\delta=0.5\%$ noise. The reconstructed amplitude is also fairly accurate for both cases but there is a larger error at the discontinuity point for case (ii). The convergence plot (mid column) indicates that the Levenberg-Marquadrdt algorithm converges steadily, and the amplitude $\lambda$ takes more iterations to reach convergence. The convergence plots show that, when varying the noise level $\delta$, the convergence is much faster for the source location $x$ than that for the amplitude $\lambda$, which agree well with the stability estimates in Theorem \ref{t5}: the recovery of the location $x$ is Lipschitz stable whereas the recovery of the amplitude $\lambda(t)$ is only logarithmically stable, and worse stability generally implies a slower convergence with respect to the noise level $\delta$ \cite{HuJinZhou:2025}.

\begin{figure}[htbp]
    \centering
    \includegraphics[width=0.32\linewidth]{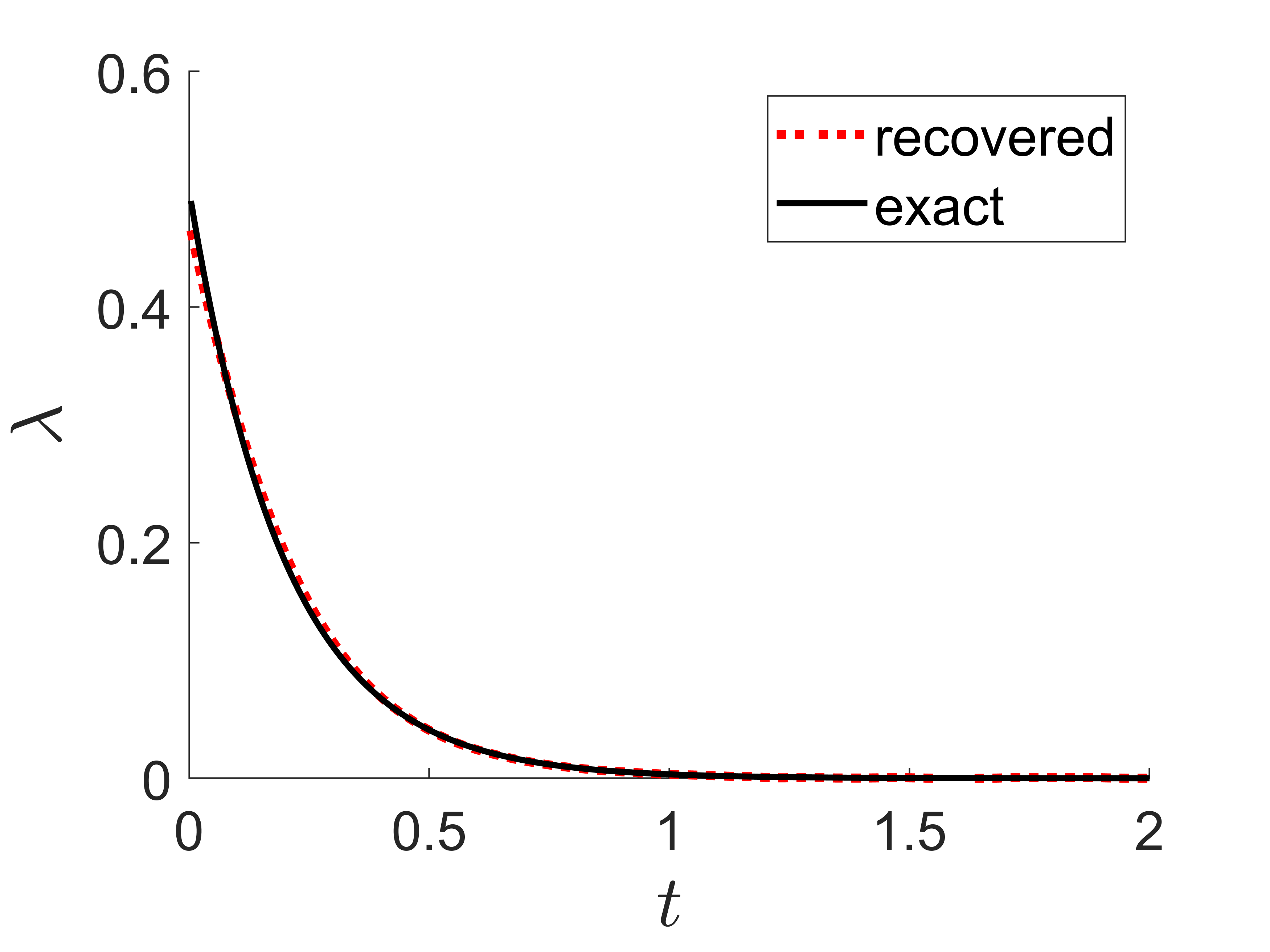}
    \includegraphics[width=0.32\linewidth]{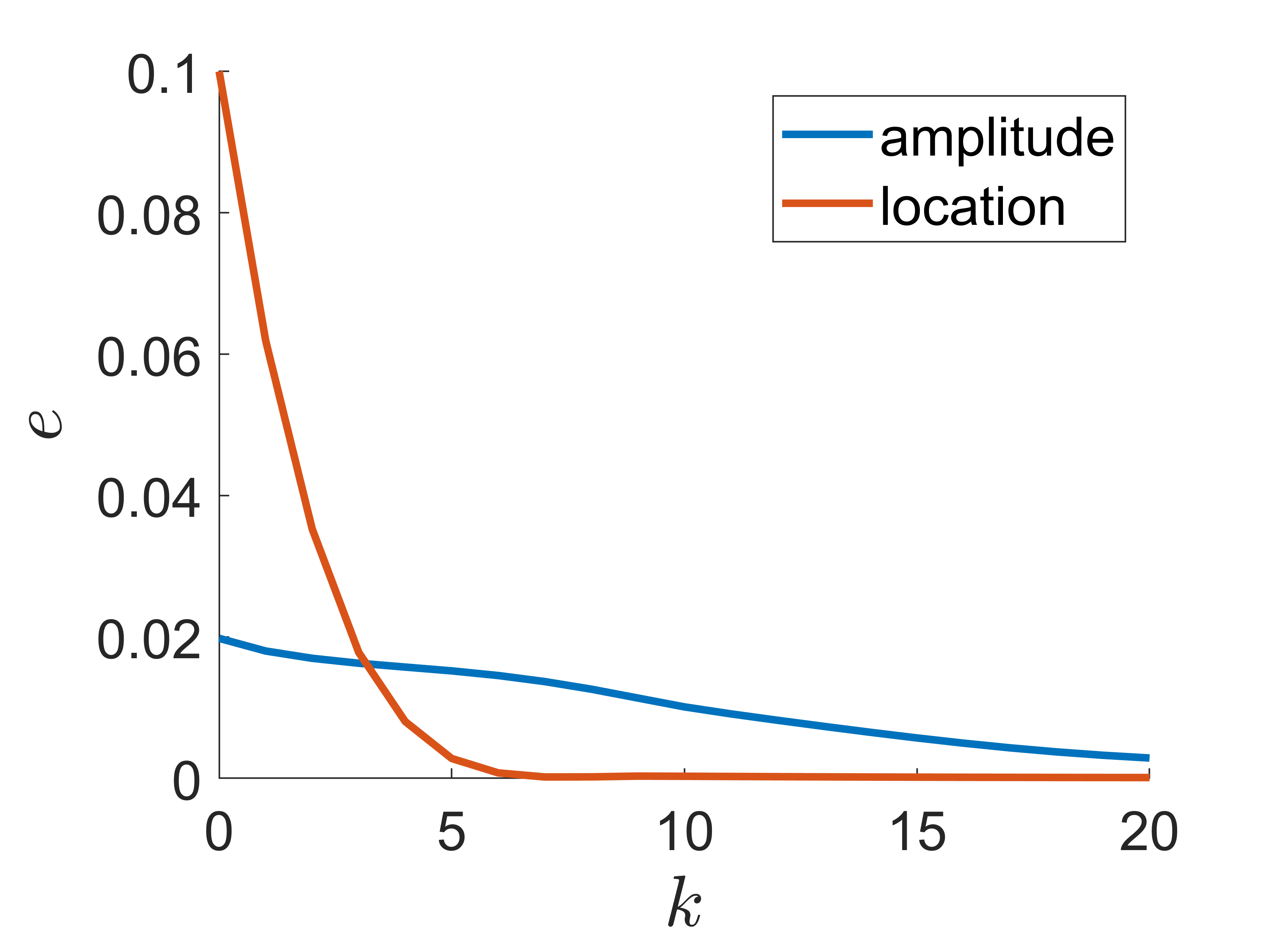}
    \includegraphics[width=0.32\linewidth]{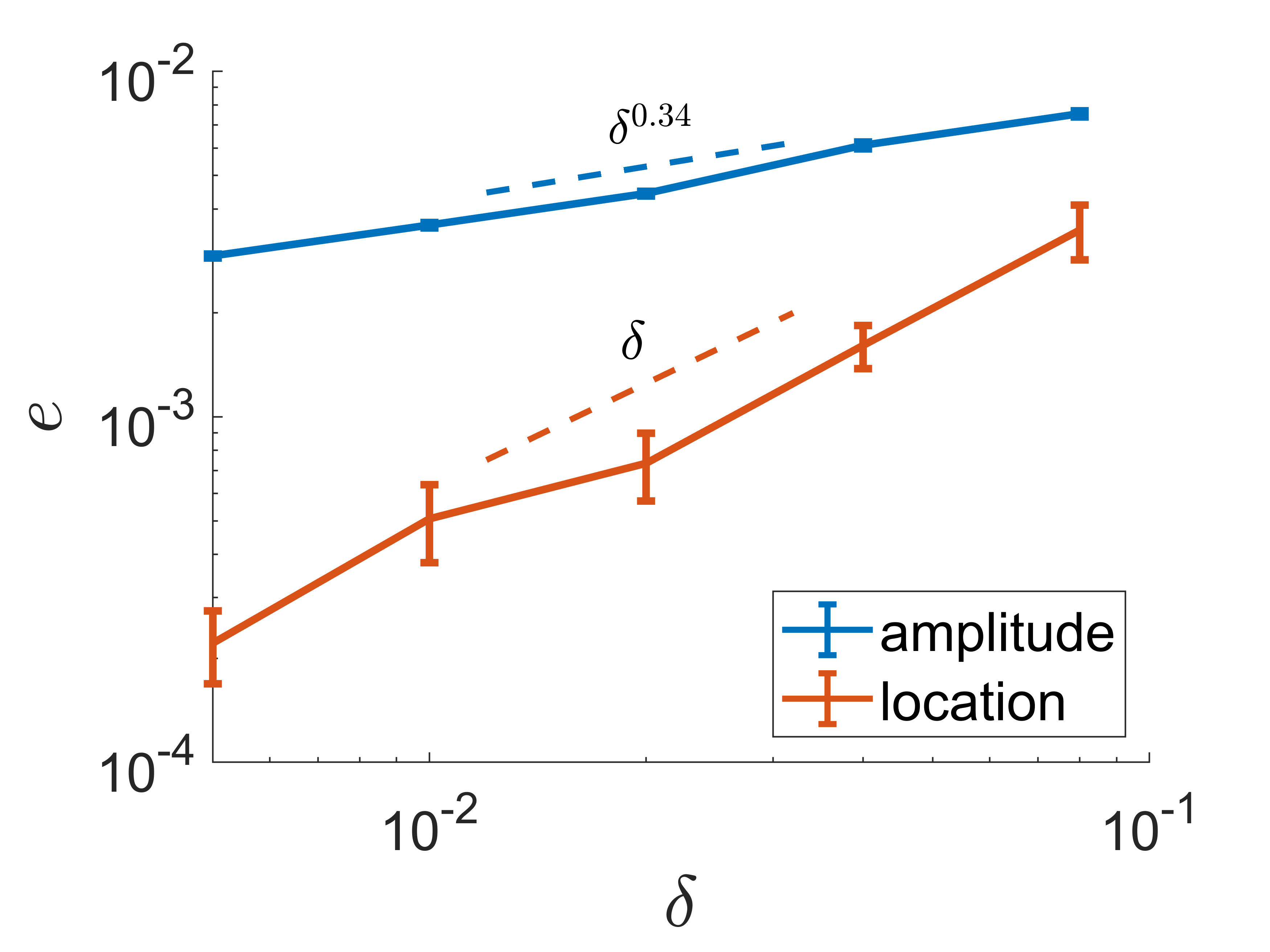}
    \includegraphics[width=0.32\linewidth]{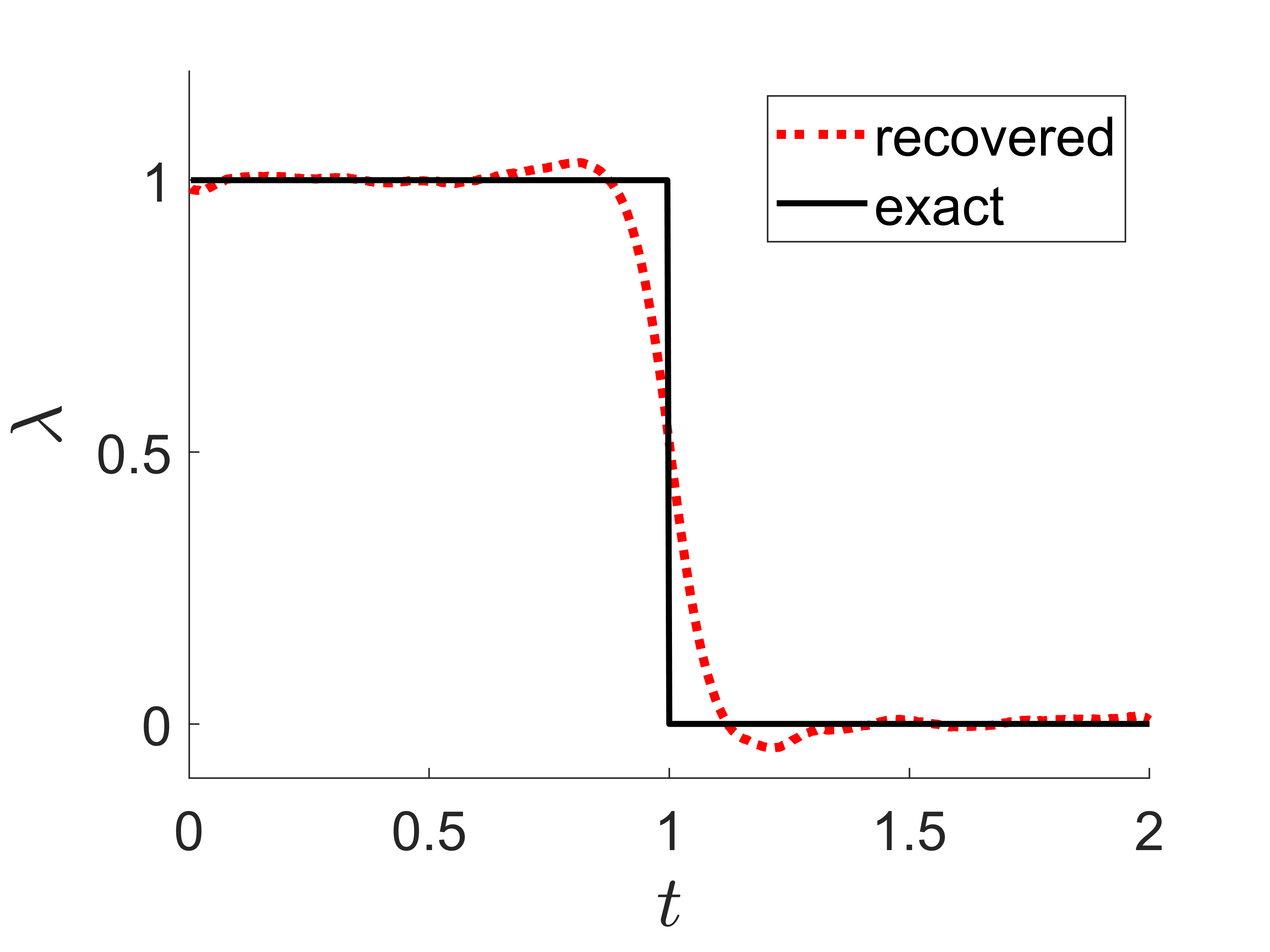}
    \includegraphics[width=0.32\linewidth]{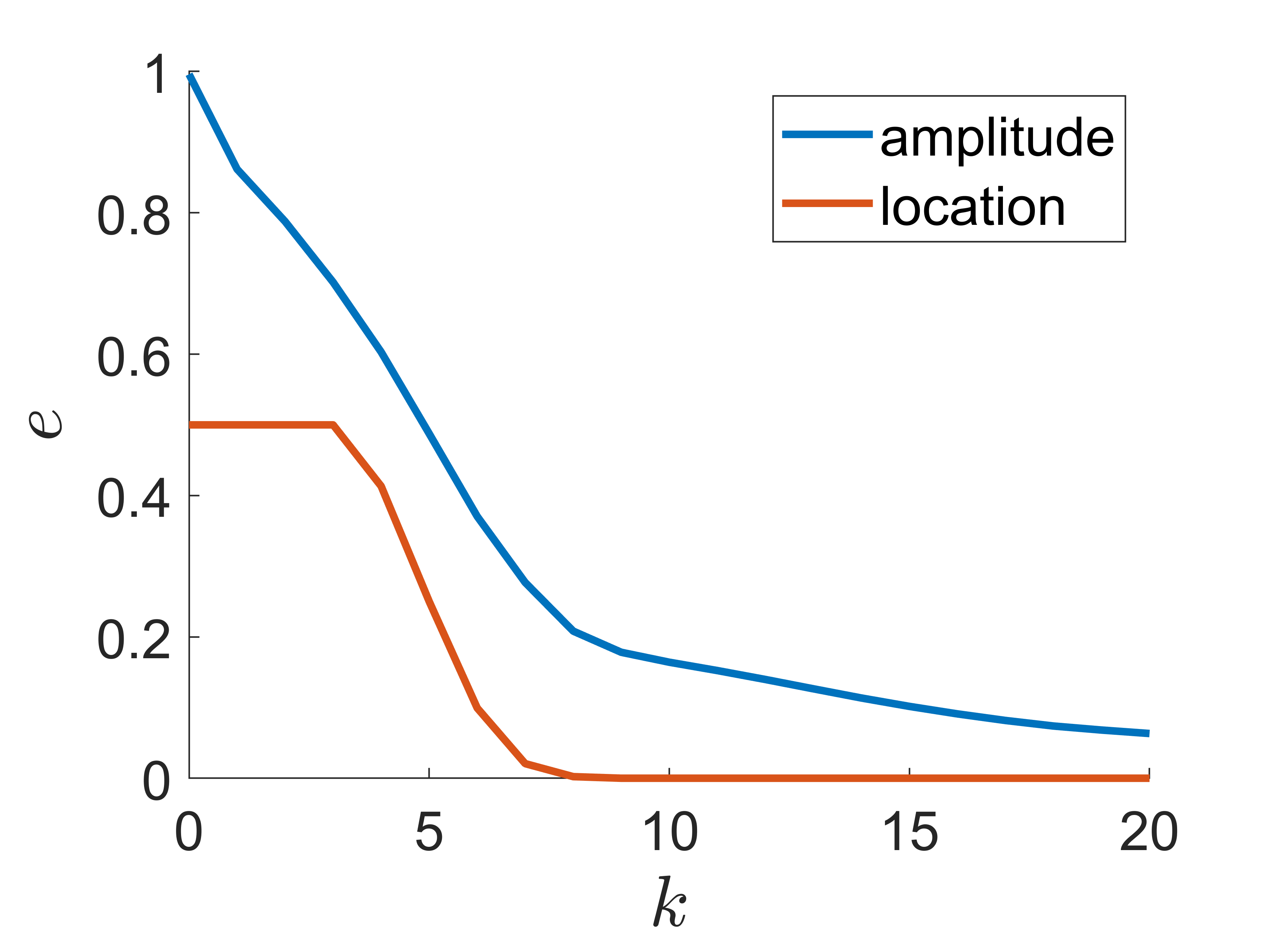}
    \includegraphics[width=0.32\linewidth]{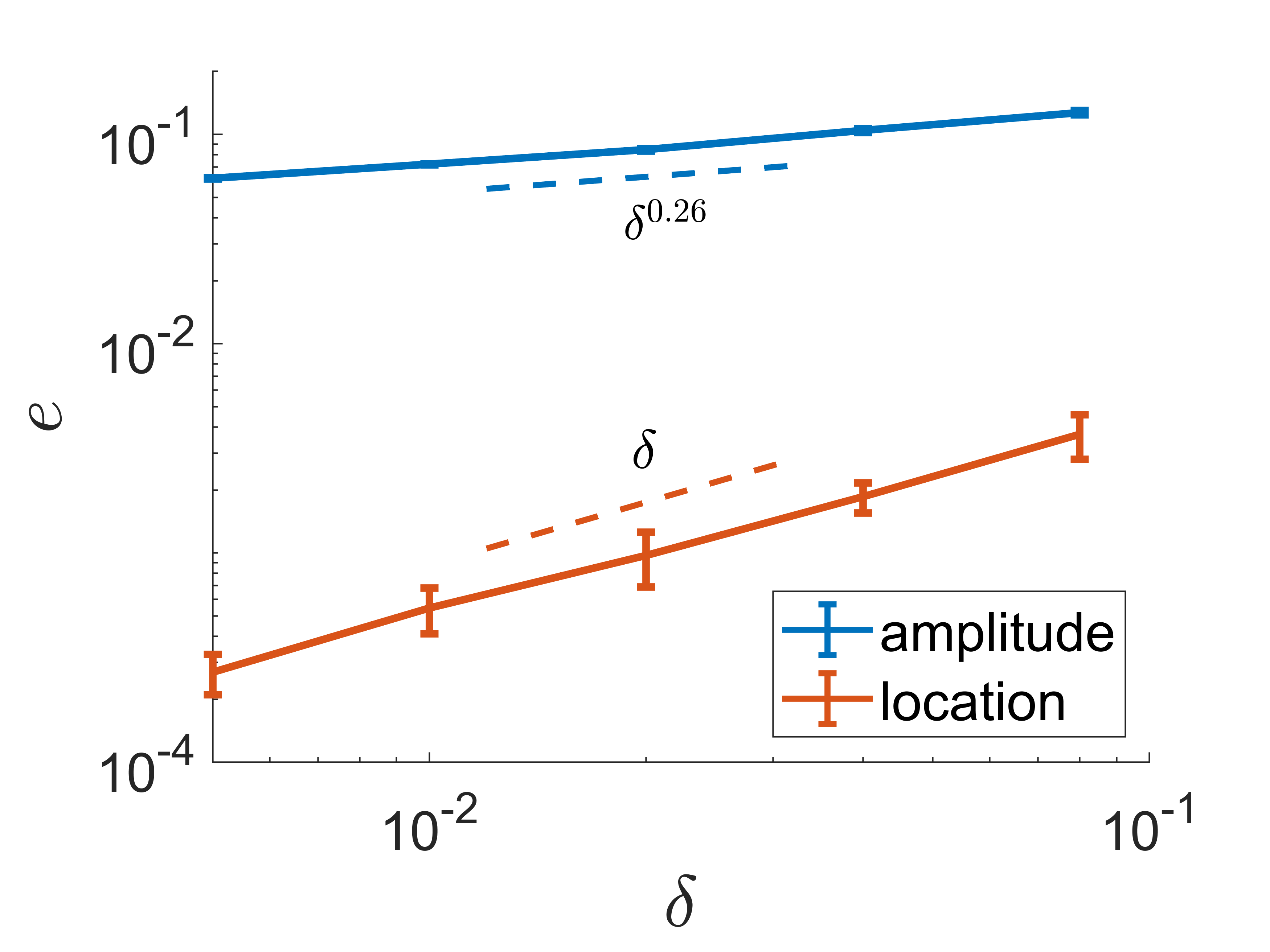}
    \caption{The numerical results for Example~\ref{ex:1} (i) (top) and (ii) (bottom): the recovered source amplitude $\lambda$ (for $\delta=0.5\%$) (left),  the error $e$ versus iteration $k$ (for  $\delta=0.5\%$), and the error $e$ versus the noise level $\delta$ on a log-log scale (right).}
    \label{fig:1}
\end{figure}

The next three examples are for two-dimensional problems. The first example deals with one point source.

\begin{example}\label{ex:2}
$\Omega=(0,1)^2$, $T=2$, $A=0$, $\mu=1$, and $u_0\equiv0$. There is a single point source at $x=(0.5,0.5)$ with amplitude given by (i) $\lambda(t)=0.5e^{-5t}$ and (ii) $\lambda(t)=\mathds{1}_{[0,1]}(t)$.
\end{example}

For case (i), the algorithm is initialized to $x^0=0.4$ and $\lambda^0(t)=0.8\lambda(t)$ for $t\in[0,2]$; for case (ii), we initialize with $x^0=0$ and $\lambda^0(t)=0$ for $t\in[0,2]$. The numerical results are presented in Fig.~\ref{fig:2}. The overall observations are similar to that for Example \ref{ex:1}. In particular the convergence rate of the reconstruction error for the amplitude $\lambda(t)$ with respect to the noise level $\delta$ is much slower than that for the location, and a nearly linear convergence rate is observed for the location $x$, which agrees well with the theoretical prediction in Theorems \ref{t3} and \ref{t100}. It is worth noting that the linear convergence in $\delta$ for the recovered location is observed for both cases (i) and (ii), despite their difference in smoothness.

The next example involves two point sources with smooth amplitudes.
\begin{example}\label{ex:3}
$\Omega=(0,1)^2$, $T=2$, $A=0$, $\mu=1$, and $u_0\equiv0$. There are two point sources at $x_1=(0.25,0.25)$ and $x_2=(0.75,0.75)$, with amplitudes $\lambda_1(t)=0.5e^{-5t}$ and $\lambda_2(t)=0.25e^{-4t}$, respectively.
\end{example}

The method is initialized to $x_1^0=(0.2,0.2)$, $x_2^0=(0.8,0.8)$, and $\lambda_i^0(t)=0.8\lambda_i(t)$ for $i=1,2$ and $t\in[0,2]$. These results are shown in Fig.~\ref{fig:3}. In the presence of two point sources, the recovery of the location becomes more challenging, and the convergence rate in terms of the noise level $\delta$ also decreases, cf. the bottom right panel of Fig. \ref{fig:3}. This observation agrees with the theoretical prediction by Theorem \ref{t3}: the stability estimate for the recovery of the location deteriorates from a Lipschitz to a H\"{o}lder one when the number of point sources increases from one to multiple. This is also confirmed by the observation that the empirical rate still agrees well with the theoretical one.

The last example involves two point sources with nonsmooth amplitudes.
\begin{example}\label{ex:4}
$\Omega=(0,1)^2$, $T=2$, $A=0$, $\mu=1$, and $u_0\equiv0$. There are two point sources at $x_1=(0.25,0.25)$ and $x_2=(0.75,0.75)$, with amplitudes $\lambda_1(t)=\mathds{1}_{[0,\frac23]}(t)$ and $
\lambda_2(t)=\mathds{1}_{[0,\frac43]}(t)$, respectively.
\end{example}

The algorithm is initialized to $x_1^0=x_2^0=(0,0)$ and $\lambda_1^0(t)=0,\lambda_2^0(t)=1$ for $t\in[0,2]$.
These results are shown in Fig.~\ref{fig:4}. Due to the presence of the discontinuous amplitudes, the convergence of the algorithm becomes less steady, and the error of the amplitudes no longer decreases monotonically during the iteration like in earlier examples. Moreover, the recoveries of the amplitudes suffer from more pronounced oscillations, which indicates more severely ill-posed nature of the inverse problem. Nevertheless, the convergence rate with respect to the noise level $\delta$ still agrees with the theoretical prediction by Theorem \ref{t3}.

\begin{figure}[htbp]
    \centering
    \includegraphics[width=0.32\linewidth]{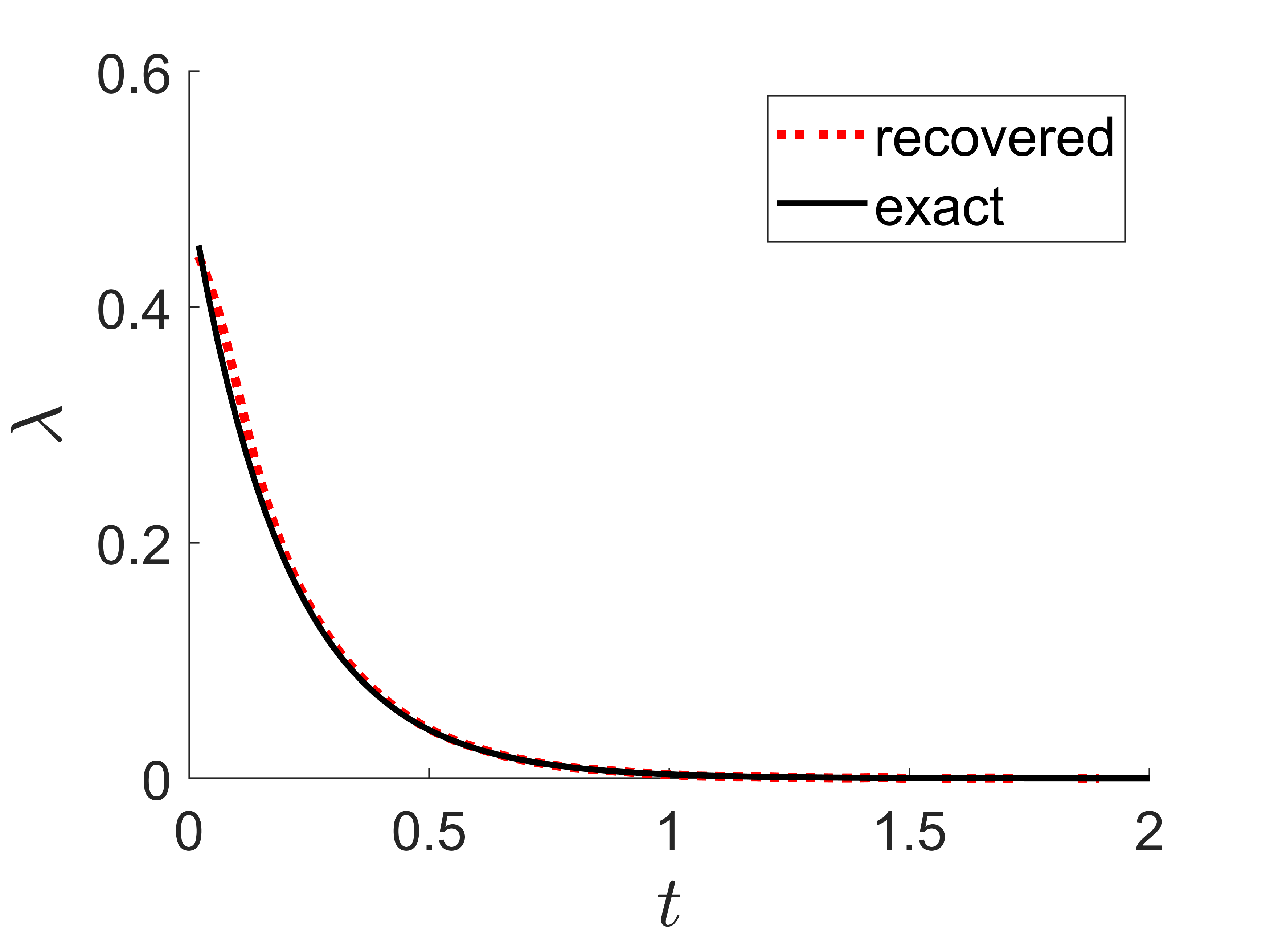}
    \includegraphics[width=0.32\linewidth]{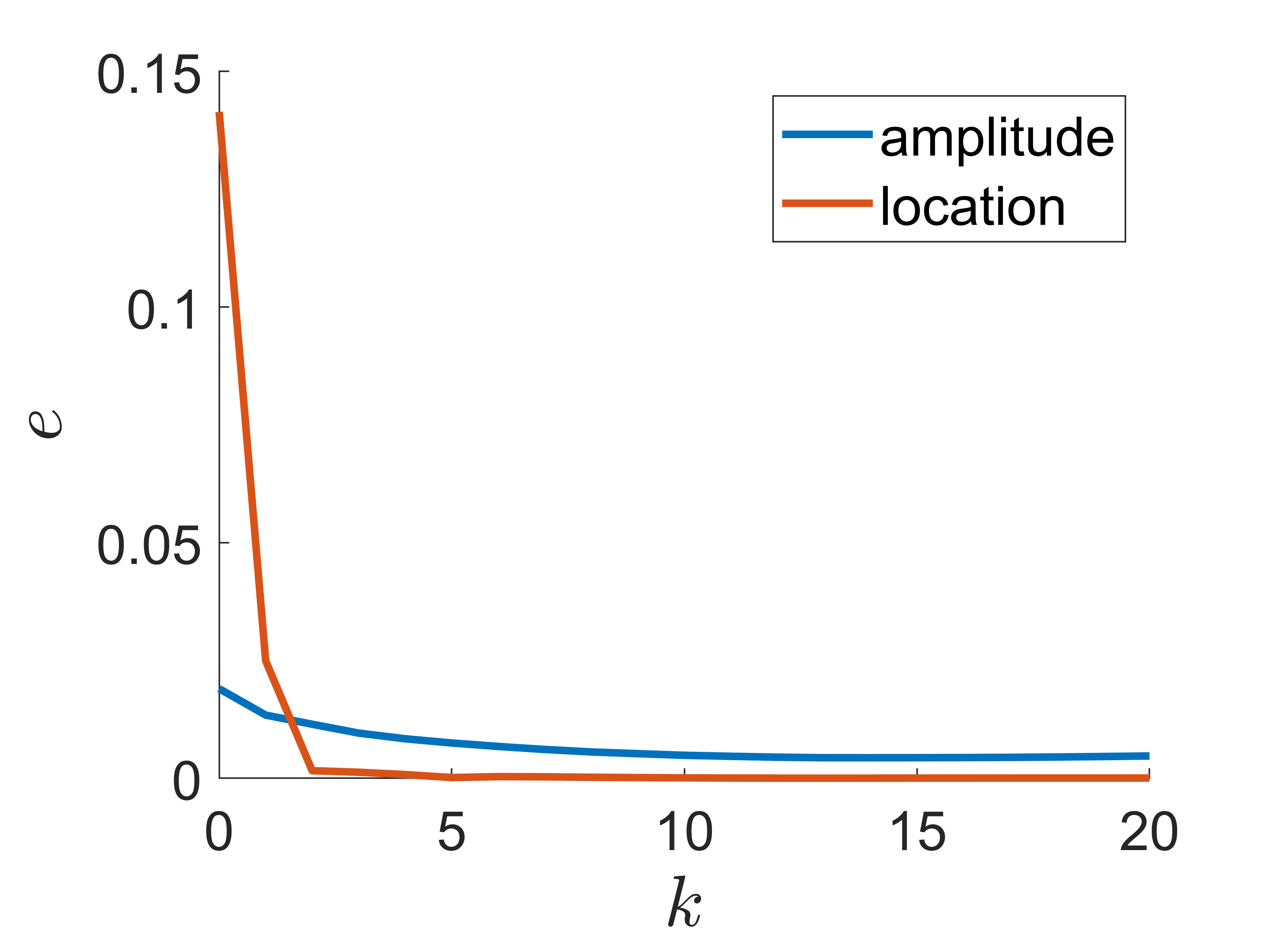}
    \includegraphics[width=0.32\linewidth]{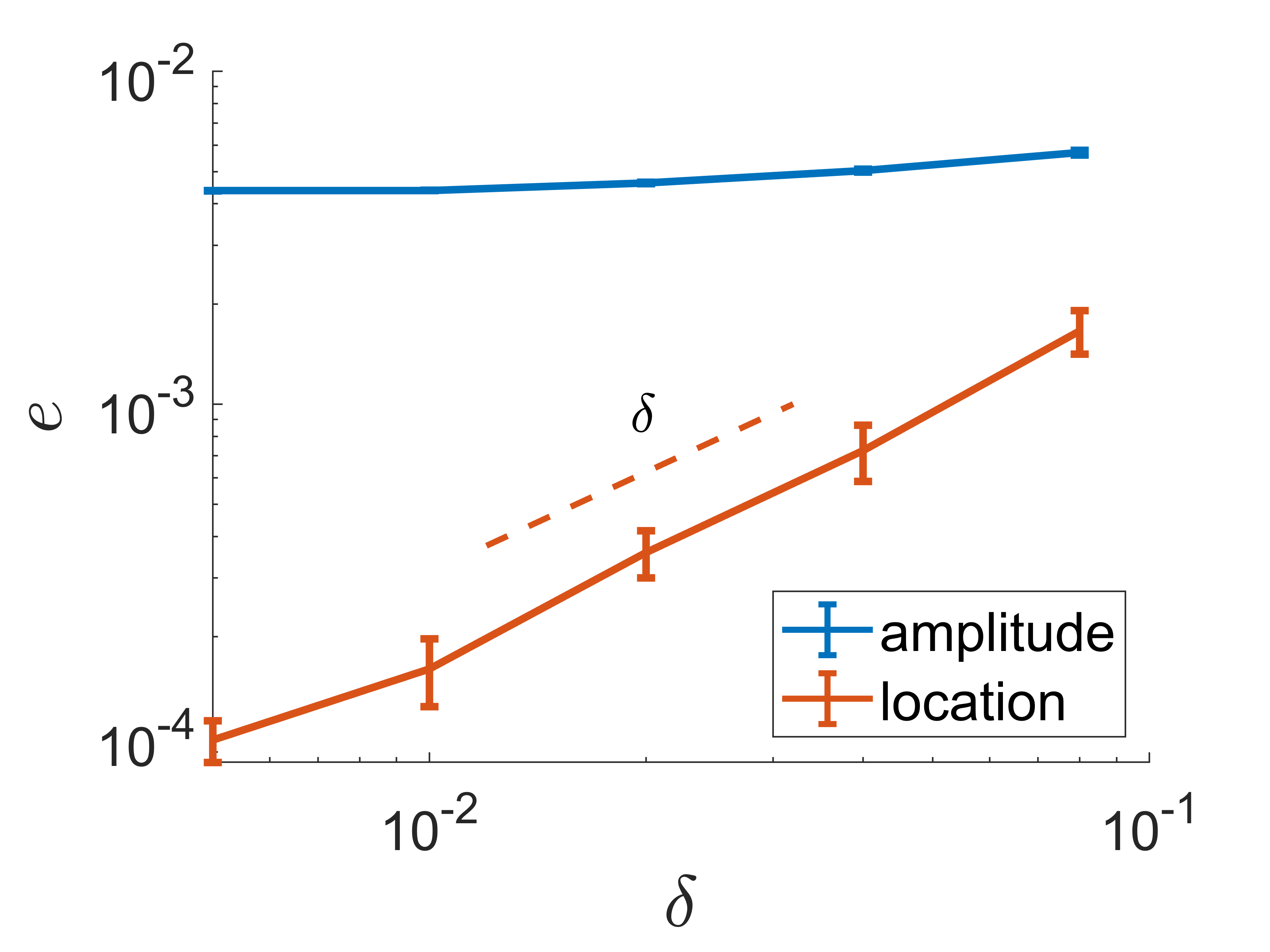}
    \includegraphics[width=0.32\linewidth]{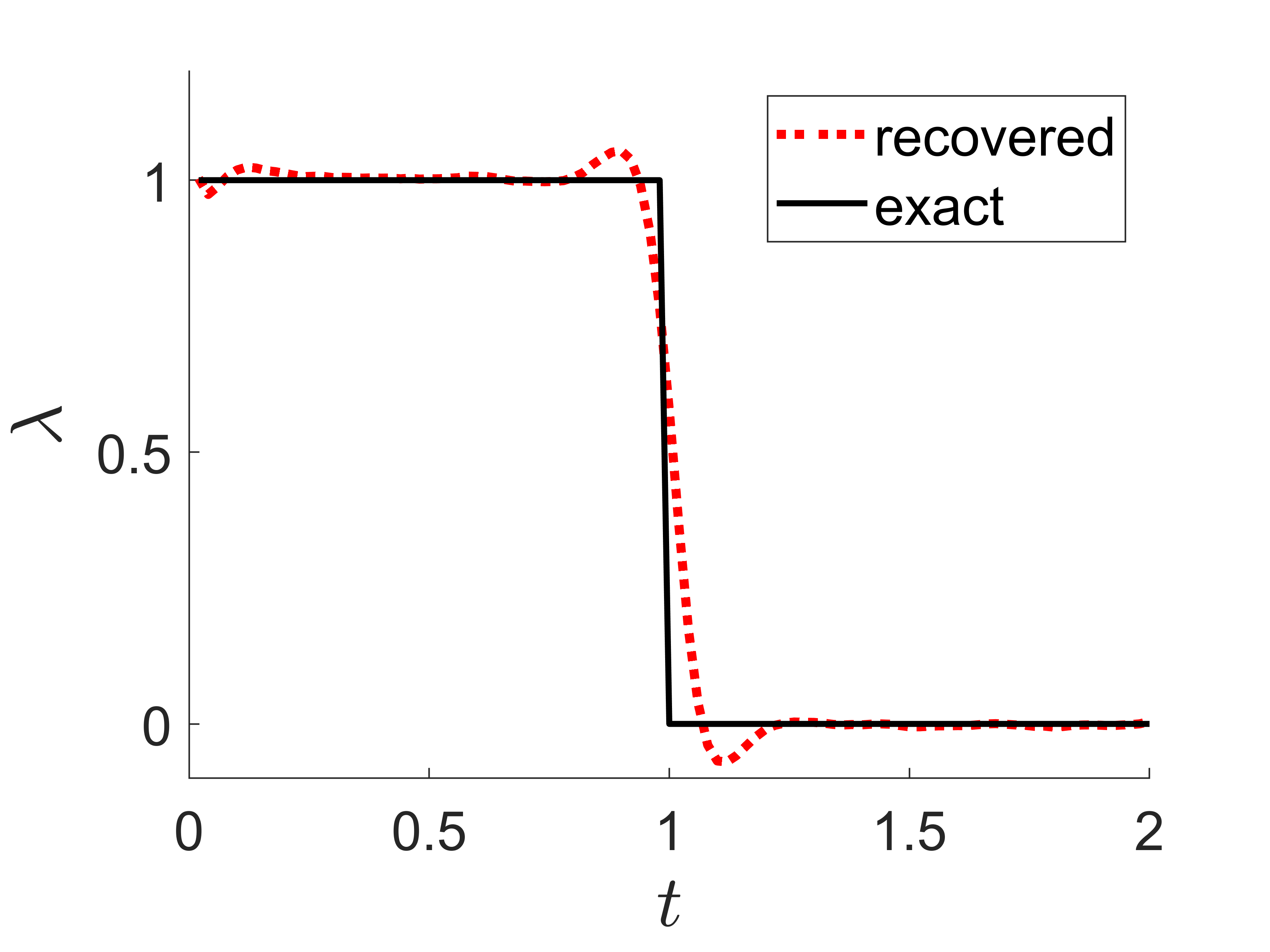}
    \includegraphics[width=0.32\linewidth]{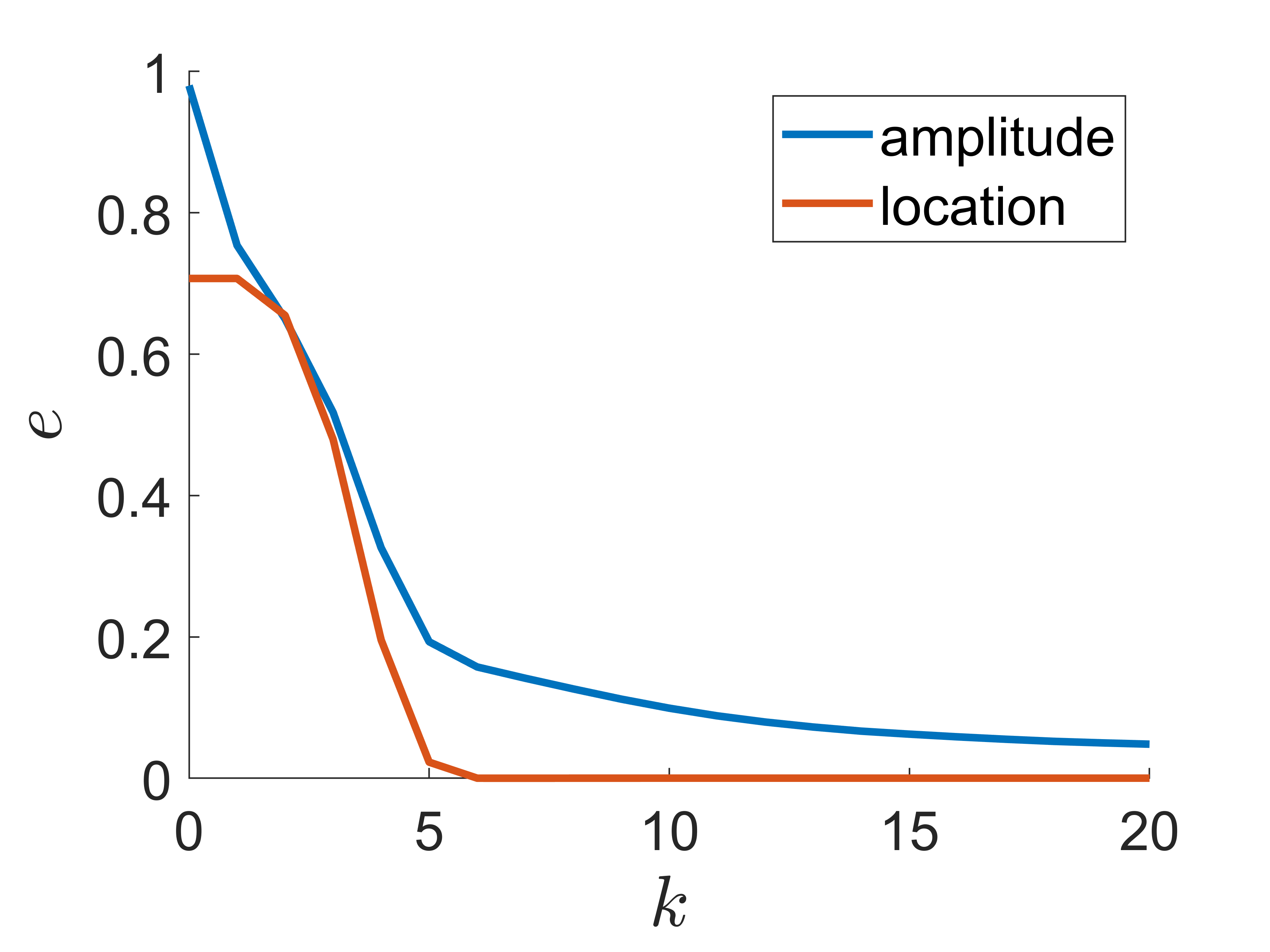}
    \includegraphics[width=0.32\linewidth]{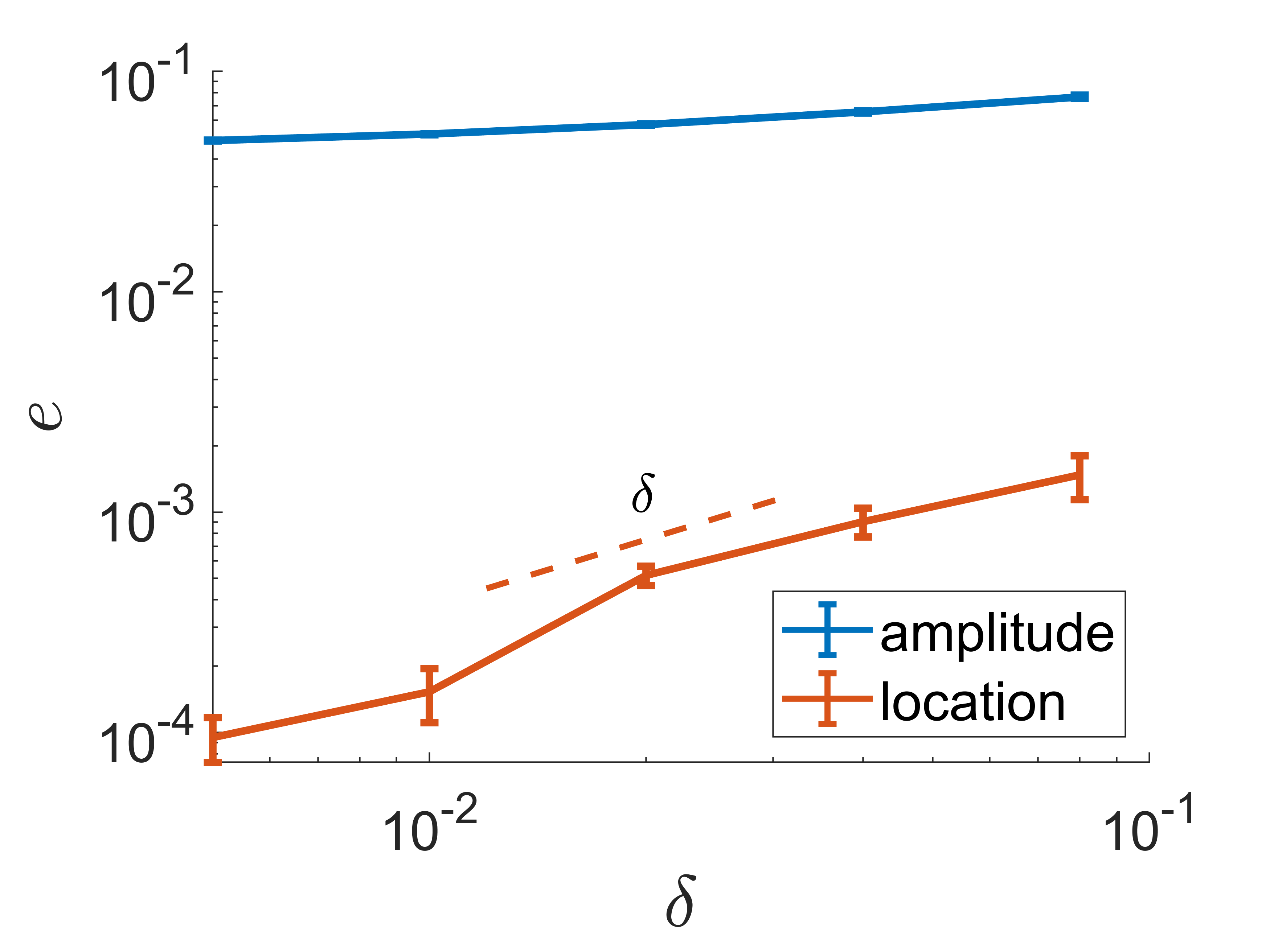}
    \caption{The numerical results for  Example~\ref{ex:2} (i) (top) and (ii) (bottom):  the recovered source amplitude (for $\delta=0.5\%$) (left); the error $e$ versus  iteration $k$ (for $\delta=0.5\%$) (middle), and the error $e$ versus the noise level $\delta$ on a log-log scale (right).}
    \label{fig:2}
\end{figure}

\begin{figure}[htbp]
    \centering
    \includegraphics[width=0.32\linewidth]{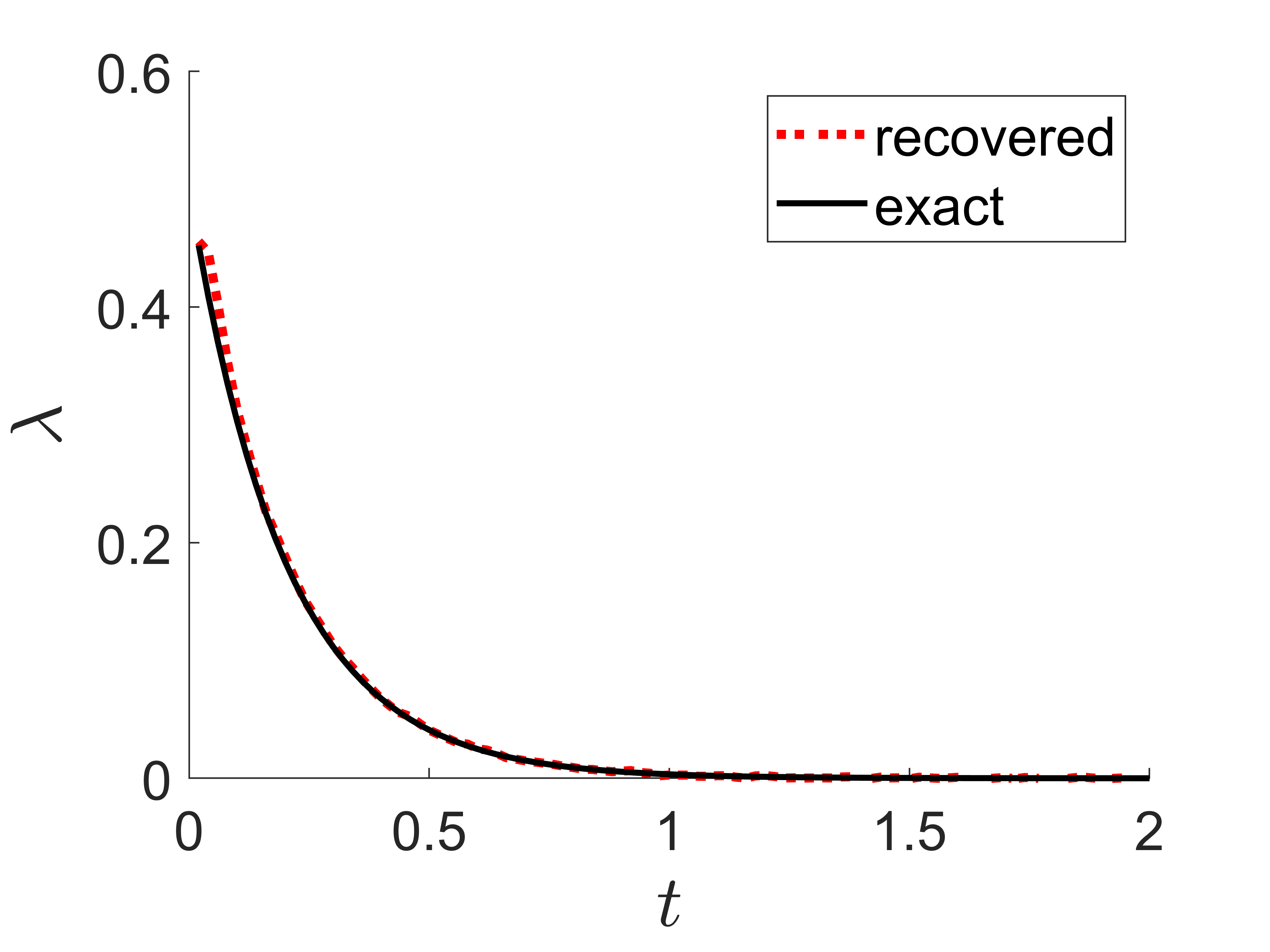}
    \includegraphics[width=0.32\linewidth]{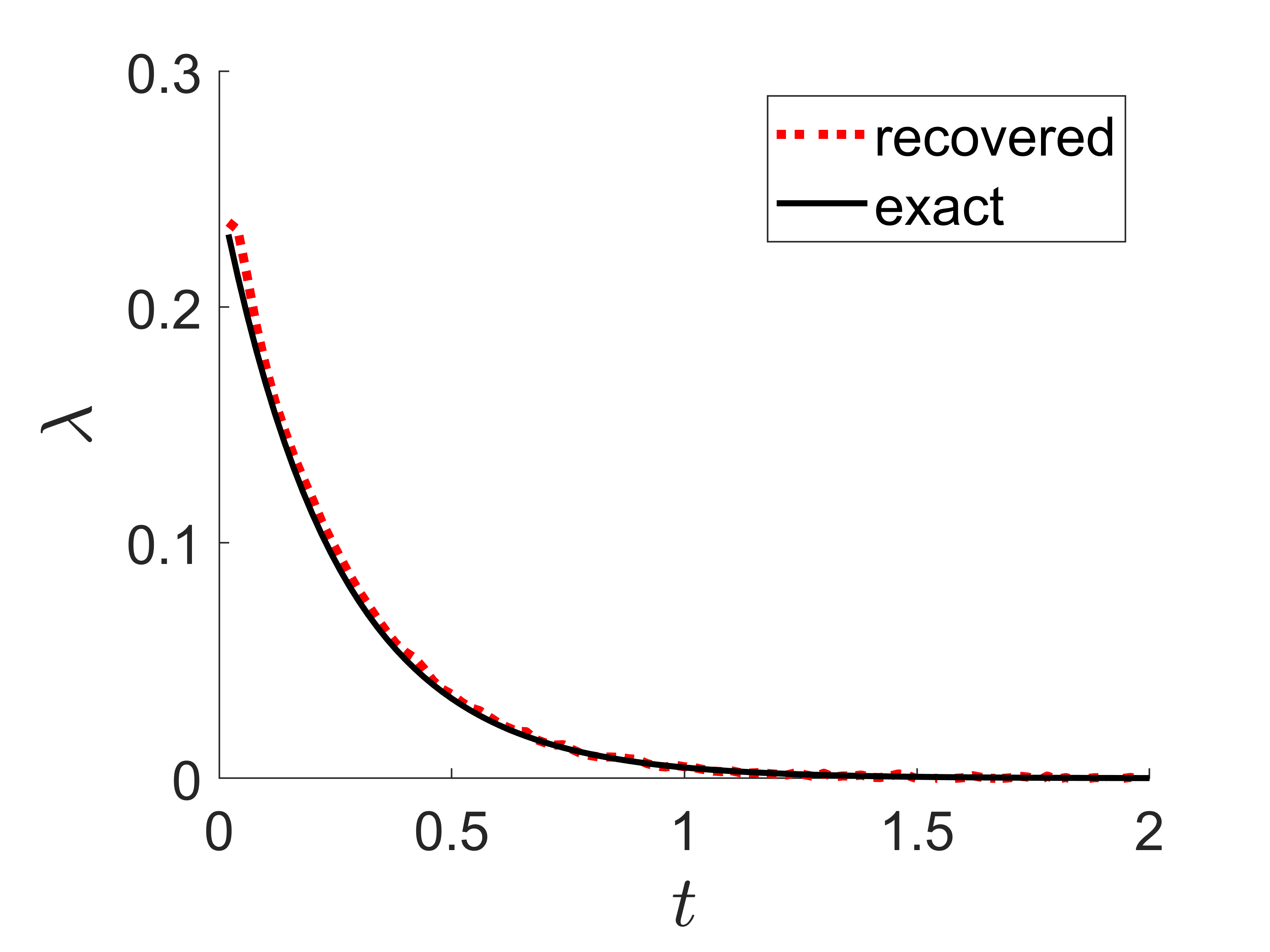} \\
    \includegraphics[width=0.32\linewidth]{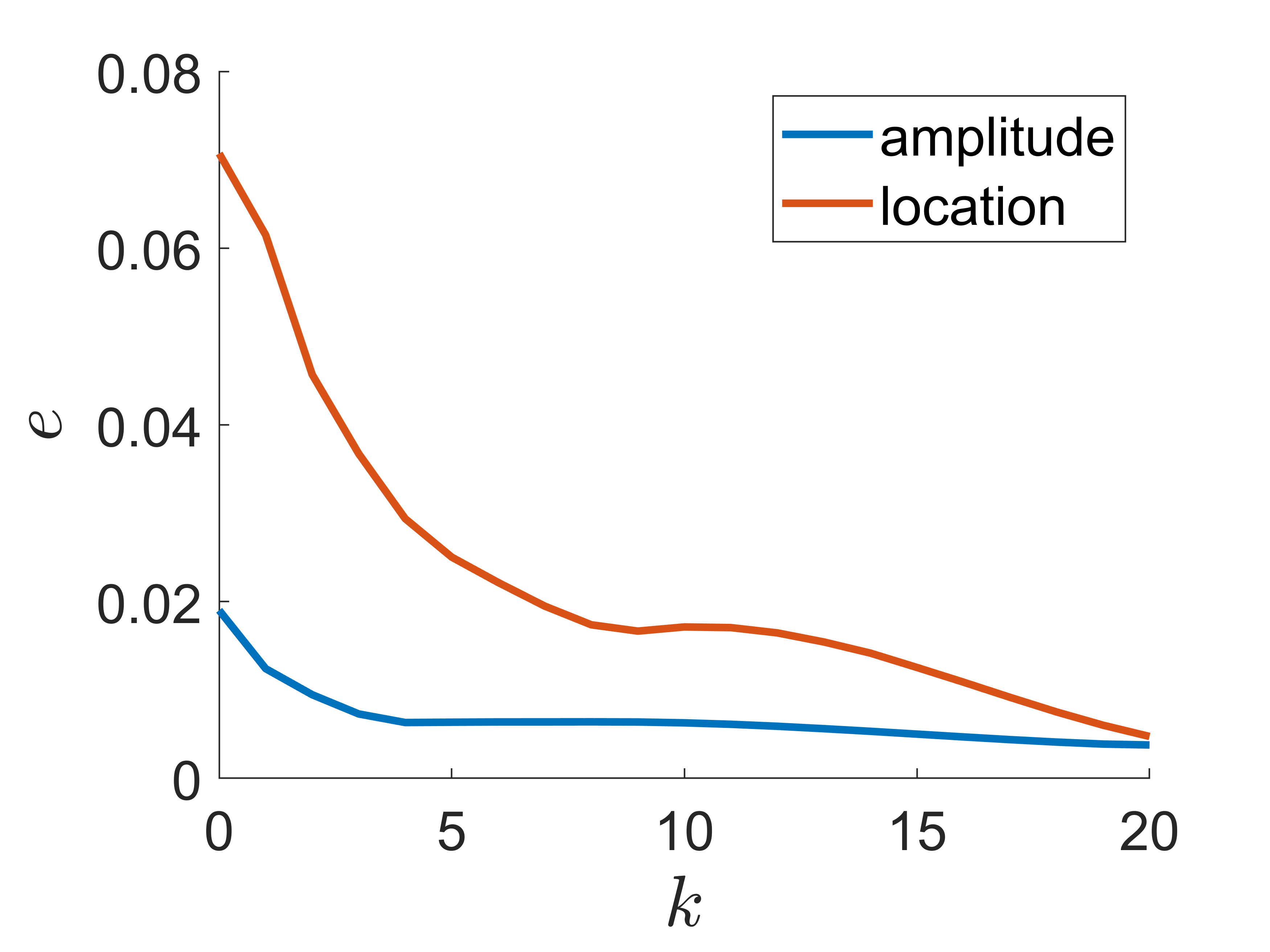}
    \includegraphics[width=0.32\linewidth]{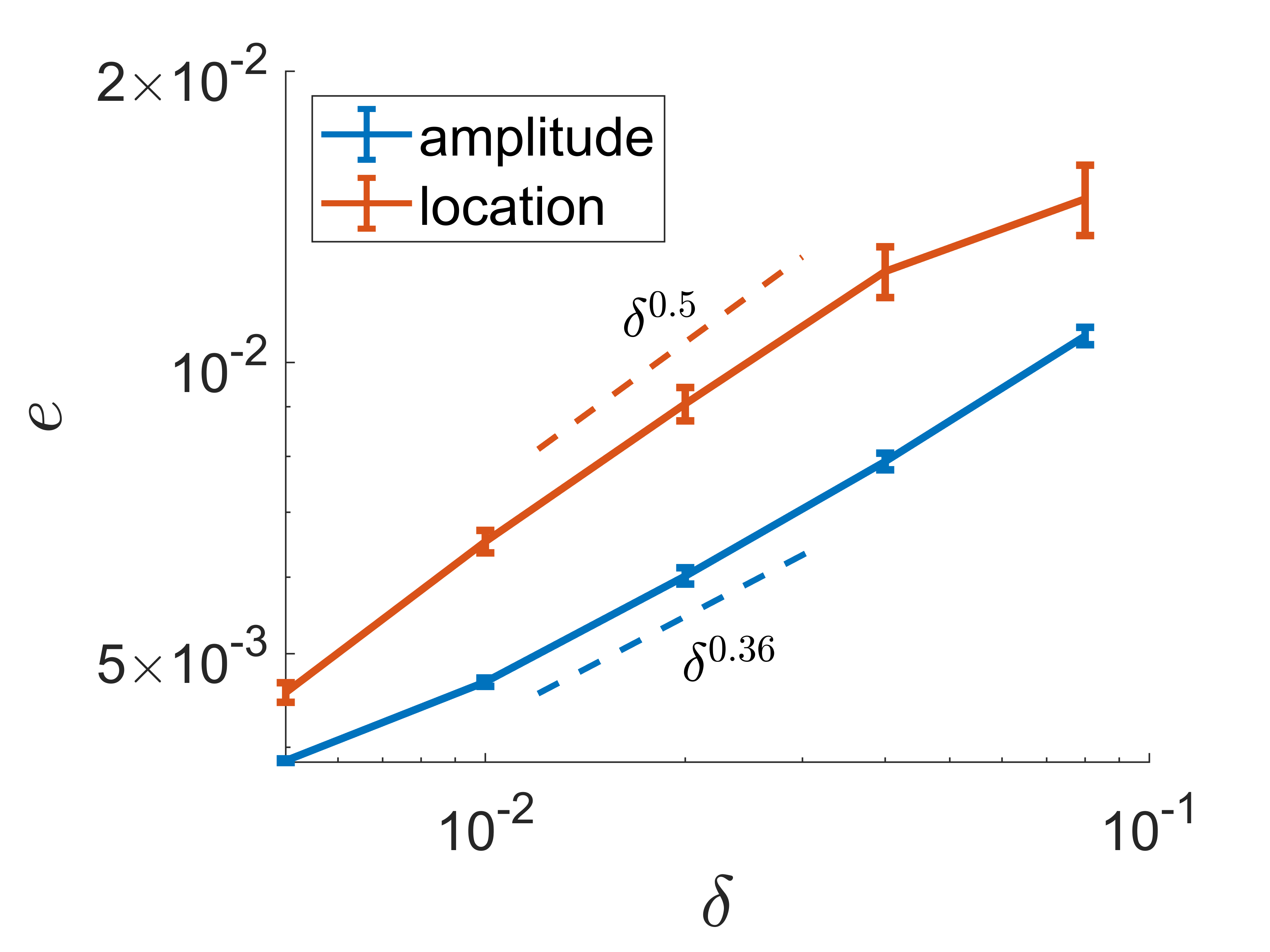}
    \caption{The numerical results for  Example~\ref{ex:3}: the recovered source amplitudes (for $\delta=0.5\%$) (top), the error $e$ versus iteration $k$ (for $\delta=0.5\%$) (bottom left), and the error $e$ versus the noise level $\delta$ on log-log scale (bottom right).}
    \label{fig:3}
\end{figure}

\begin{figure}[htbp]
    \centering
    \includegraphics[width=0.32\linewidth]{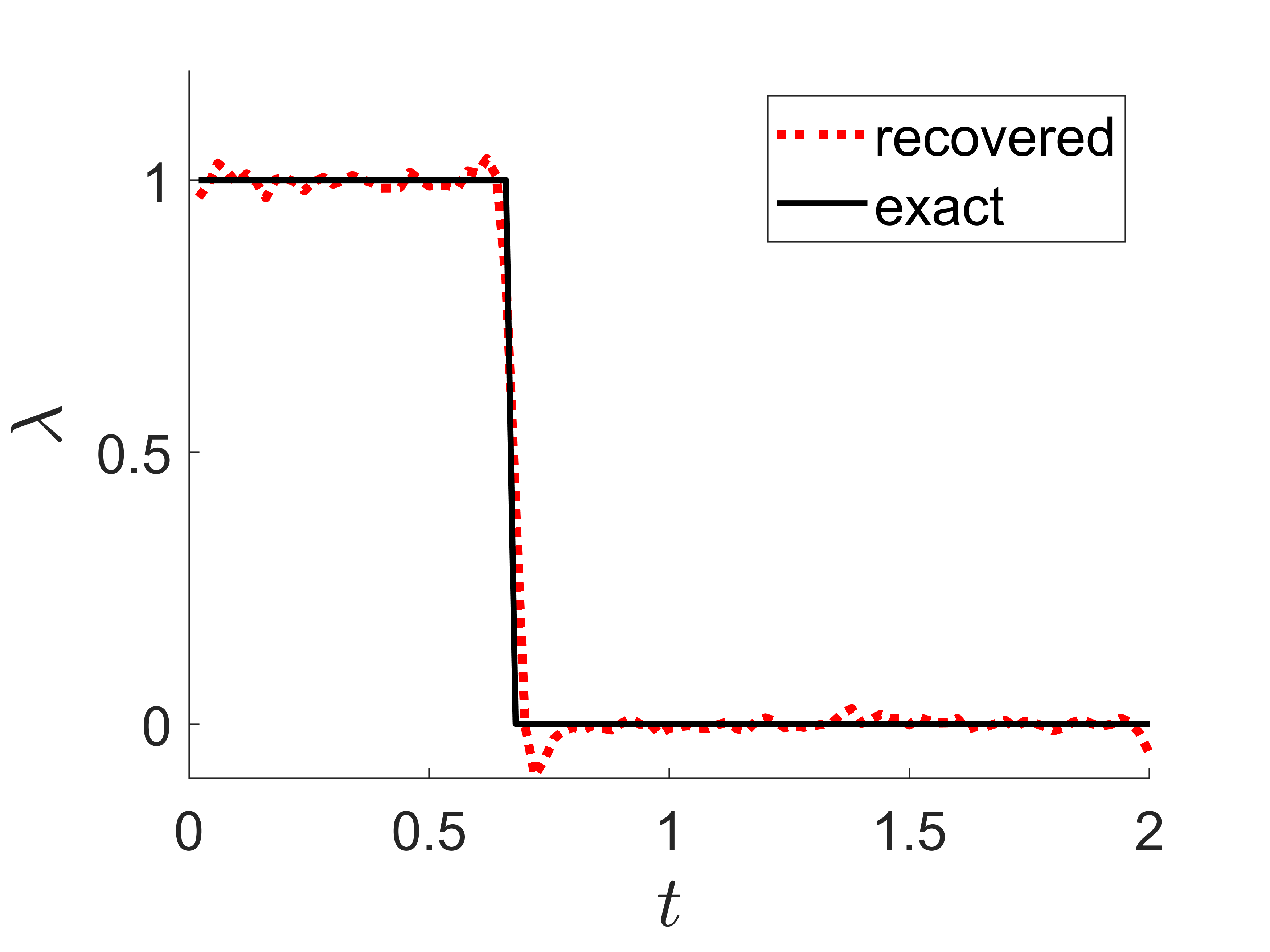}
    \includegraphics[width=0.32\linewidth]{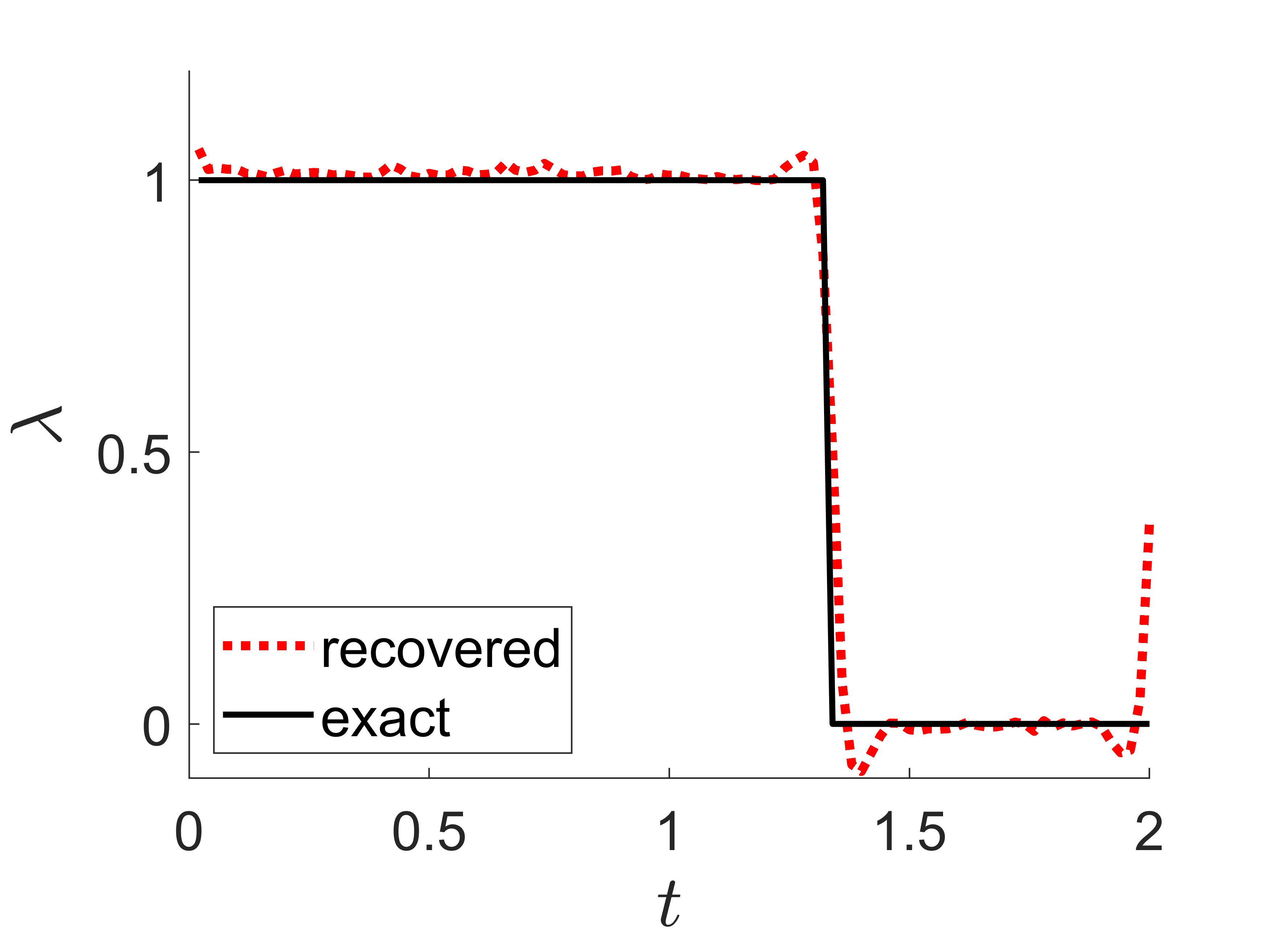} \\
    \includegraphics[width=0.32\linewidth]{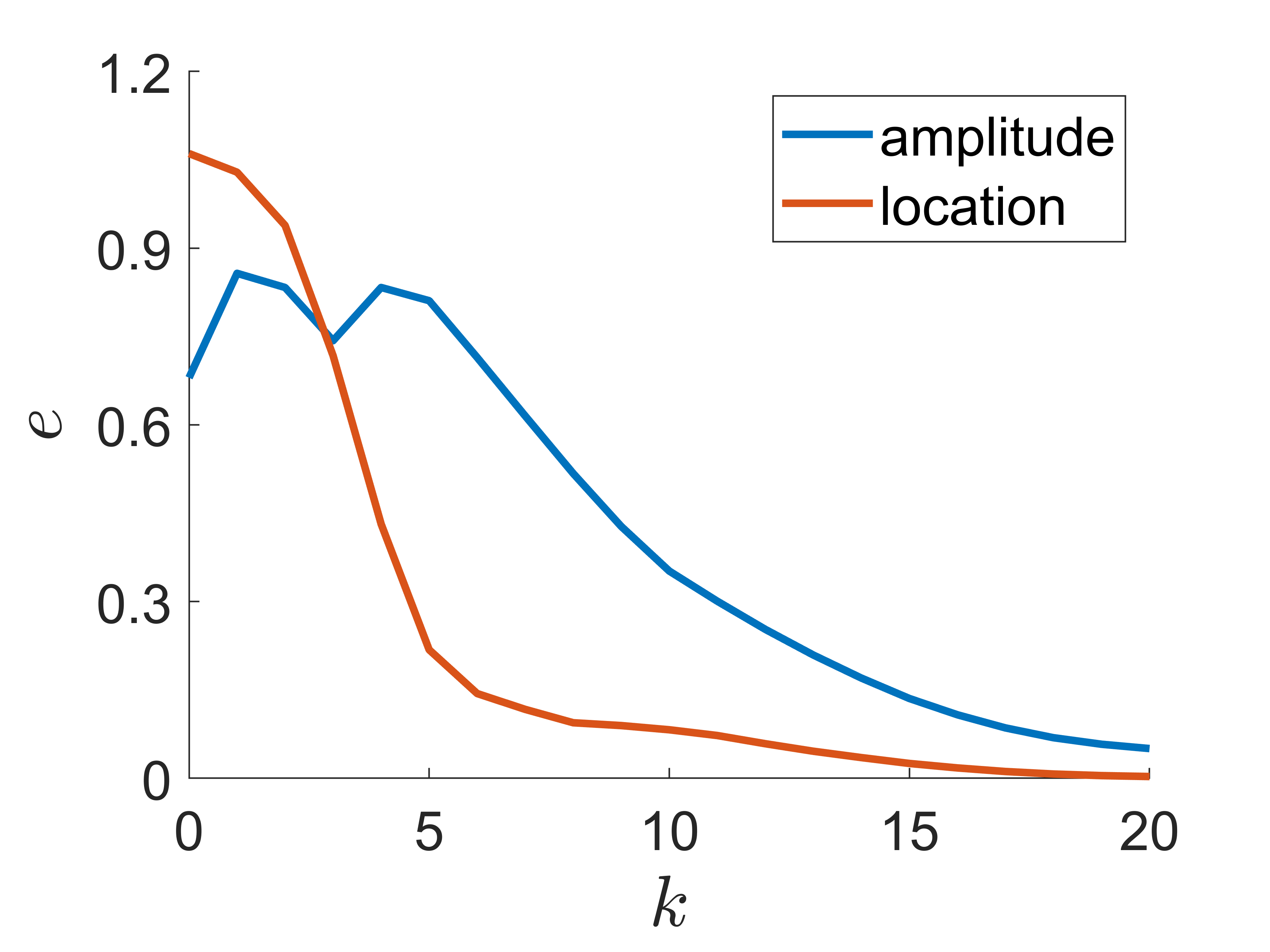}
    \includegraphics[width=0.32\linewidth]{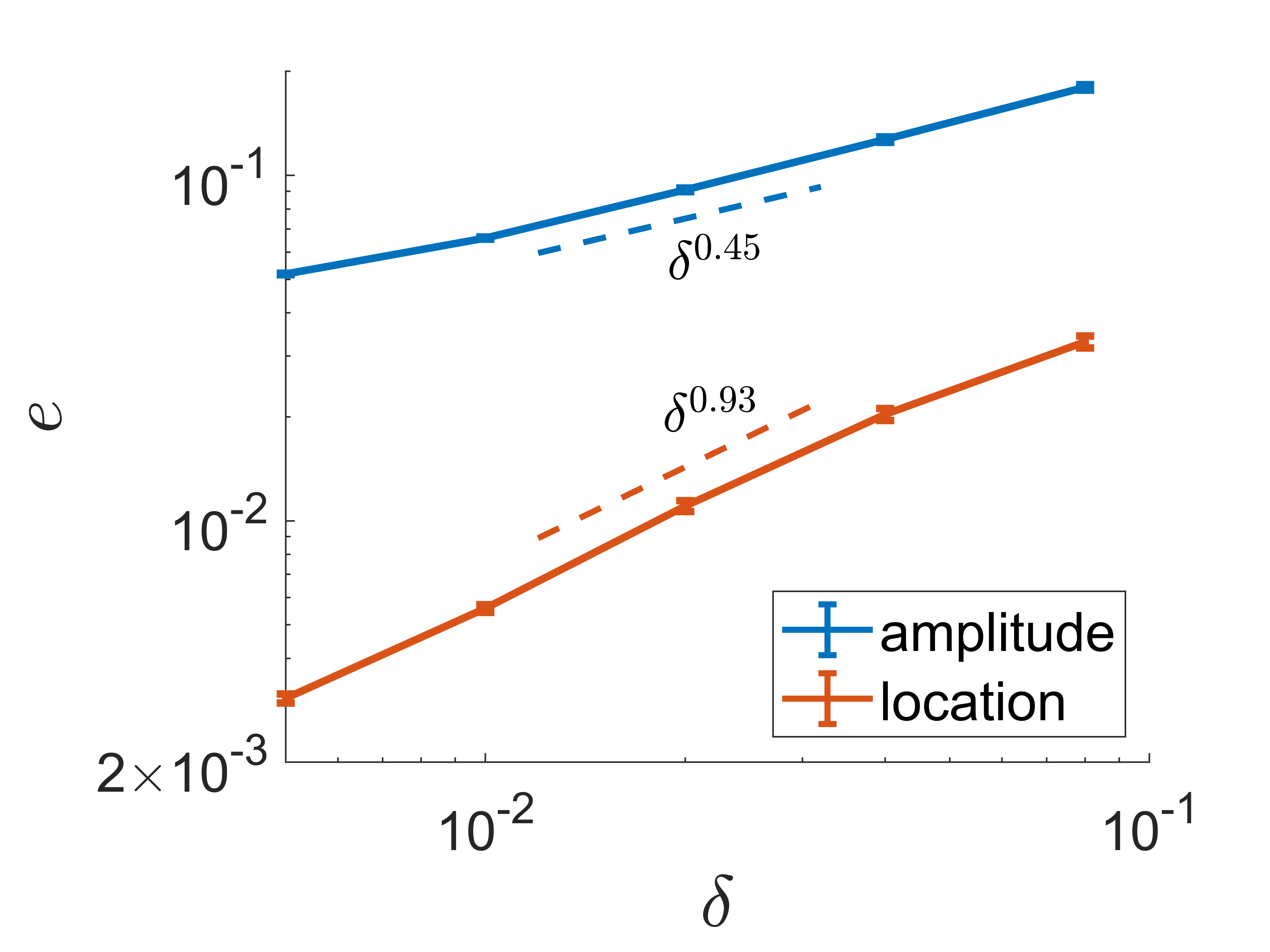}
    \caption{Numerical results for Example~\ref{ex:4}:  the recovered source amplitudes (for  $\delta=0.5\%$) (top), the error $e$ versus iteration $k$ (for $\delta=0.5\%$) (bottom left), and the error $e$ versus the noise level $\delta$ on a log-log scale (bottom right).}
    \label{fig:4}
\end{figure}

\section{Conclusion}
In this work we have established several new stability estimates for the inverse problem of recovering point sources in the parabolic equation. The stability estimate is of Lipschitz / H\"{o}lder type for recovering the location of the point source, but it is of logarithmic type for recovering the amplitude. The theoretical findings are supported by numerical experiments. Future research problems including multiple point sources and partial Cauchy data, as well as the theoretical analysis of the reconstruction algorithms.
\bibliographystyle{abbrv}
\bibliography{point}

\end{document}